\documentclass[10pt]{amsart}

\usepackage{amssymb,amsmath,amsfonts,amsthm,amscd}

\usepackage{graphicx}
\usepackage{color}
\usepackage{xspace}

\usepackage{txfonts}
\usepackage{mathrsfs}

\newtheorem{theorem}{Theorem}[section]
\newtheorem{proposition}[theorem]{Proposition}    
\newtheorem{lemma}[theorem]{Lemma}           
\newtheorem{corollary}[theorem]{Corollary}
\newtheorem{assumption}[theorem]{Assumption}

\theoremstyle{definition}

\newtheorem{remark}[theorem]{Remark}

\newcommand{\mb}[1]{\ensuremath{\bf{#1}}}
\newcommand{\N}{{\mb{N}}}

\newcommand{\R}{{\mb{R}}}

\newcommand{\al}{\alpha}
\newcommand{\be}{\beta}

\newcommand{\A}{{\mathscr A}}
\newcommand{\B}{{\mathscr B}}

\newcommand{\W}{{\mathcal W}}
\renewcommand{\P}{{\mathfrak P}}
\newcommand{\M}{{\mathcal M}}

\newcommand{\K}{{\mathcal K}}
\newcommand{\I}{{\mathcal I}}
\newcommand{\J}{{\mathcal J}}

\let \Re \relax
\DeclareMathOperator{\Re}{Re}
\let \Im \relax
\DeclareMathOperator{\Im}{Im}

\newcommand{\ovl}[1]{\overline{#1}}

\newcommand{\Con}{\ensuremath{\mathscr C}}
\newcommand{\Cinf}{\ensuremath{\Con^\infty}}
\newcommand{\Cinfc}{\ensuremath{\Con^\infty}_{c}}
\renewcommand{\S}{\ensuremath{\mathscr S}}

\newcommand{\inp}[2]{\langle #1, #2 \rangle} 
\newcommand{\scp}[2]{\left( #1 , #2 \right)}

\newcommand{\norm}[2]{{\| #1 \|}_{#2}}
\newcommand{\sob}[2]{\ensuremath{(H^{#1},H^{#2})}}

\newcommand{\est}[1]{\langle #1 \rangle}

\newcommand{\symp}[4]{
\sigma((#1,#2),(#3,#4))}
\newcommand{\sympl}[4]{
 \inp{#2}{#3} - \inp{#4}{#1}}

\newcommand{\Op}{\ensuremath{\mathrm{Op}}}

\DeclareMathOperator{\id}{Id}
\DeclareMathOperator{\supp}{supp}
\renewcommand{\d}{\ensuremath{\partial}}

\newcommand{\mint}[1]{\int\limits_{\bigcirc\hspace{-0.45em}\text{\tiny$#1$}}}
\newcommand{\transp}{\ensuremath{\phantom{}^{\displaystyle{t}}}}

\newcommand{\wt}{\ensuremath{\widetilde}}
\renewcommand{\qedsymbol}{$\blacksquare$}

\newcommand{\bld}[1]{\mbox{\boldmath $#1$}}

\newcommand{\mstrut}[1]{\mbox{\rule{0mm}{#1}}}

\newcommand{\hf}{\frac{1}{2}}

\newcommand{\wrt}{w.r.t.\@\xspace}
\newcommand{\rhs}{r.h.s.\@\xspace}

\newcommand{\ie}{i.e.\@\xspace}
\newcommand{\eg}{e.g.\@\xspace}
\newcommand{\resp}{resp.\@\xspace}

\newcommand{\nhd}{neighborhood\xspace}

\numberwithin{equation}{section}

\begin{document}
\baselineskip 16pt

\date{\today}
% title
\title[Pseudodifferential multi-product]{Pseudodifferential multi-product
      representation of the solution operator
      of a  parabolic equation}

\author{Hiroshi Isozaki}
\address{University of Tsukuba, Institute of Mathematics,
  1-1-1 Tennodai Tsukuba Ibaraki 305-8571, Japan}
\author{J\'{e}r\^{o}me Le Rousseau}
\address{Universit\'{e}s
  d'Aix-Marseille, Universit\'e de Provence, Laboratoire d'Analyse
  Topologie Probabilit\'{e}s, CNRS UMR 6632, 39 rue F.~Joliot-Curie,
  13453 Marseille cedex 13, France. Current address:
  Universit\'e d'Orl\'eans,
  Laboratoire Math\'ematiques et Applications, Physique Math\'ematique
  d'Orl\'eans, CNRS UMR 6628, F\'ed\'eration Denis Poisson, FR CNRS
  2964, B.P. 6759, 45067 Orl\'eans cedex 2, France.}

% abstract
\begin{abstract} By using a time slicing procedure, we represent the
  solution operator of a second-order parabolic pseudodifferential
  equation on $\R^n$ as an infinite product of 
    zero-order pseudodifferential operators. A similar representation
  formula is proven for parabolic differential equations on a compact
  Riemannian manifold. Each operator in the multi-product is given by
  a simple explicit Ansatz.  The proof is based on an effective use of
  the Weyl calculus and the
  Fefferman-Phong inequality.   \\

  
\end{abstract}

\maketitle

%%%%%%%%%%%
% contents
%\tableofcontents

%%%%%%%%%%%%%%%%%%%%%%%%%
% Sections              %
%%%%%%%%%%%%%%%%%%%%%%%%%
\section{Introduction and notation}
We begin with recalling standard notation for the calculus of
pseudodifferential operators ($\psi$DOs). Throughout the article, we
shall most often use spaces of global symbols; a function $a \in
\Con^{\infty}(\R^n\times \R^p)$ is in $S^m(\R^n\times \R^p)$ if
for all multi-indices $\alpha$, $\beta$ there exists $C_{\alpha
  \beta}>0$ such that
\begin{equation}
  |\partial_x^\alpha \partial_{\xi}^\beta  a(x,\xi)| \leq C_{\alpha \beta}\: 
  \est{\xi}^{m  - |\beta|}, \ \ x \in \R^n, 
  \ \xi \in \R^p, \quad \est{\xi} := (1 + |\xi|^2)^{1/2}.
\end{equation}
We write $S^m = S^m(\R^n\times\R^n)$.
$\psi$DOs of order $m$, in  Weyl quantization, 
are formally given by  (see \cite{Hoermander:79} or \cite[Chapter
18.5]{Hoermander:V3})
\begin{equation*}
  \Op^w(a)\, u(x) = a^{w}(x,D_x) u (x) = (2\pi)^{-n}\iint e^{i\inp{x-y}{\xi}} a((x+y)/2,\xi)\: 
  u (y)\: d y \: d \xi,\quad u \in \S'(\R^n).
\end{equation*}
We denote by $\Psi^m(\R^n)$, or simply by $\Psi^m$, the space of such
$\psi$DOs of order $m$.

We consider a second-order $\psi$DO defined by the Weyl quantization
of $q(x,\xi)$. Assuming uniform ellipticity and
  positivity for $q(x,\xi)$, we study the following parabolic Cauchy
problem
\begin{align}
  \label{eq: cauchy pb 1 intro}
  \d_t u + q^w(t,x,D_x) u &=0, \ \ \ 0< t\leq T,\\
  u \mid_{t=0} &= u_0,
  \label{eq: cauchy pb 2 intro}
\end{align}
for $u_0$ in $L^2(\R^n)$ or in some Sobolev space. The solution operator of this Cauchy problem is denoted by $U(t',t)$, $0\leq t\leq t'\leq T$.
Here, we are interested in providing a representation
of $U(t',t)$ in the form of a multi-product of $\psi$DOs.

Such a representation is motivated by the results of the second author
in the case of hyperbolic equations
\cite{LeRousseau:04,LeRousseau:06b}. If the symbol $q$ is only a
function of $\xi$, the solution of (\ref{eq: cauchy pb 1
  intro})--(\ref{eq: cauchy pb 2 intro}) is simply given by means of Fourier
transformations as
\begin{align*}
  u(t,x) = (2\pi)^{-n} \iint e^{i\inp{x-y}{\xi}} e^{-t q(\xi)}\, 
  u (y)\: d y \: d \xi.
\end{align*}
Following \cite{LeRousseau:04}, we then hope to have a good
approximation of $u(t,x)$, for small $t$, in the case where $q$ depends
also on both $t$ and $x$: 
\begin{align*}
  u(t,x) \simeq (2\pi)^{-n} \iint e^{i\inp{x-y}{\xi}} e^{-t q(0,(x+y)/2,\xi)}\, 
  u (y)\: d y \: d \xi = p_{(t,0)}^w(x,D_x) u (x),
\end{align*}
where $p_{(t'',t')}(x,\xi) := e^{-(t''-t') q(t',x,\xi)}$, $0\leq t'\leq t''\leq T$,
which is in $S^0$. The infinitesimal approximation we introduce is thus
of pseudodifferential nature (as opposed to Fourier integral operators
in the hyperbolic case \cite{LeRousseau:04}). 

With such an infinitesimal operator, by iterations, we are then led
to introducing the following multi-product of $\psi$DOs to approximate the 
solution operator $U(t',t)$ of the Cauchy problem (\ref{eq: cauchy pb 1
  intro})--(\ref{eq: cauchy pb 2 intro}):
\begin{align*}
    \W_{\P,t} := \left\{
      \begin{array}{ll}
        p_{(t,0)}^w(x,D_x) & \text{if }\ 0\leq t\leq t^{(1)},\\
        p_{(t,t^{(k)})}^w(x,D_x) {\displaystyle \prod_{i=k}^{1}} 
	p_{(t^{(i)},t^{(i-1)})}^w(x,D_x) 
        & \text{if }\ t^{(k)}\leq t\leq t^{(k+1)}.
      \end{array}
    \right.
  \end{align*}
  where $\P=\{t^{(0)},t^{(1)},\dots,t^{(N)}\}$ is a subdivision of
  $[0,T]$ with $0=t^{(0)}< t^{(1)}<\dots <t^{(N)}=T$. It should be
  noted that in general $p_{(t'',t')}^w(x,D_x)$, $t'\leq t''$, does not have  semi-group properties.

In \cite{LeRousseau:04}, for the hyperbolic case, the standard
quantization is used for the equation and for the approximation
Ansatz. However, in the present parabolic case this approach fails
(see Remark~\ref{rem: jutification Weyl} below). Instead, the choice of
Weyl quantization yields convergence results of the
Ansatz $\W_{\P,t}$ comparable to those in
\cite{LeRousseau:04,LeRousseau:06b}. The convergence of
$\W_{\P,t}$ to the solution operator $U(t,0)$ is shown in operator norm with
an estimate of the convergence rate depending on the (H\"older)
regularity of $q(t,x,\xi)$ \wrt the evolution parameter $t$. See
Theorem~\ref{theorem: main theorem Rn} in Section~\ref{sec:
  convergence} below for a precise statement.

Such a convergence result thus yields a representation of the solution
operator of the Cauchy problem (\ref{eq: cauchy pb 1
  intro})--(\ref{eq: cauchy pb 2 intro}) by an infinite multi-product
of $\psi$DOs. The result relies (i) on the proof of the stability of
the multi-product $\W_{P,t}$ as $N= |\P|$ grows to $\infty$
(Proposition~\ref{prop:Hs norm under control}) and (ii) on a
consistency estimate that measures the infinitesimal error made by replacing
$U(t',t)$ by $p_{(t',t)}^w(x,D_x)$ (Proposition~\ref{prop: consistency
  error}). The stability in fact follows from a sharp Sobolev-norm estimate for $p_{(t',t)}^w(x,D_x)$ (see Theorem~\ref{theorem: sharp estimate}): for $s\in\R$, there exists $C\geq 0$ such that 
\begin{align}
  \label{eq: intro sharp estimate}
  \| p_{(t',t)}^{w}(x,D_x)\|_{(H^s,H^s)} \leq 1 + C (t'-t).
\end{align}
The Fefferman-Phong inequality plays an important role here.

The representation of the solution operator by multi-products of
$\psi$DO follows from the exact convergence of the Ansatz $\W_{\P,t}$
in some operator norm. We emphasize that the convergence we obtain is
not up to a regularizing operator. A further interesting aspect of
this result is that each constituting operator of the multi-product is
given explicitly. With such a product representation, we have in
mind the development of numerical schemes for practical
applications. Once the problem is discretized in space, the use of
fast Fourier transformations (FFT) can yield numerical methods with
low computational complexity, with possibly microlocal approximations
of the symbols in question as is sometimes done in the case of
hyperbolic equations (see for instance
\cite{dHlRW:00,lr01a,lr01b,LRDH:03}). We also anticipate that our
representation procedure can be used in theoretical purposes.

As described above, the first part of this article is devoted to the
parabolic Cauchy problem on $\R^n$ and to the study of the properties
of the approximation Ansatz $\W_{\P,t}$. In the second part, we shall
consider a parabolic problem on a compact Riemannian manifold without
boundaries. In this case, the operator $q^w(x,D_x)$ is only considered
of differential type, in particular for its full symbol to be known
exactly.  In each local chart we can define an infinitesimal
approximation of the solution operator as is done in $\R^n$ and we
combine these local $\psi$DOs together with the help of a partition of
unity. This yields a counterpart of $p_{(t',t)}^w(x,D_x)$ for the
manifold case (see Section~\ref{secM: intro}) denoted by $P_{(t',t)}$.
In fact, the sharp estimate~(\ref{eq: intro sharp estimate}) still
holds in $L^2$ ($s=0$) for this infinitesimal approximation
(Theorem~\ref{theoremM: sharp estimate}).  The proof of a consistency
estimate (Proposition~\ref{prop: consistency manifold}) requires the
analysis of the effect of changes of variables for Weyl symbols of the
form of $e^{-(t'-t) q}$. The choice we have made for the definition of
$P_{(t',t)}$ is invariant through such changes of variables up to a
first-order precision \wrt the small parameter $h=t'-t$, which is
compatible with the kind of results we are aiming at.  With stability
and consistency at hand, the convergence result then follows as in the
case of $\R^n$.

 In the manifold case, the constituting $\psi$DOs of the
  multi-product are given explicitly in each local chart.  We observe
  moreover that the computation of the action of these local operators
  can be essentialy performed as in the case of $\R^n$, which is appealing 
for practical implementations.

Another approach to representation of the solution operator $U(t',t)$
can be found in the work of C.~Iwasaki (see
\cite{Tsutsumi:74,Iwasaki:77,Iwasaki:84}). Her work encompasses the
case of degenerate parabolic operators, utilizes multi-product of
$\psi$DOs and analyses the symbol of the resulting operator, using the
work of Kumano-go \cite{Kumano-go:81}. However, the symbol of the
solution operator $U(t',t)$ is finally obtained by solving a Volterra
equation. Such integral equations also appear in related works on the
solution operator of parabolic equations (see \eg
\cite{Greiner:71,ST:84}).  The alternative method we present here will be
more suitable for applications because of the explicit aspect of the
representations. The step of the integral equation in the above works
makes the representation formula less explicit. However, the reader
will note that the technique we use in our approach here do
not apply to the case of degenerate parabolic equations like those
treated in \cite{Tsutsumi:74,Iwasaki:77,Iwasaki:84}. The question of
the extension of the convergence and representation results we present
here to the case of degenerate parabolic equations appears to us an
interesting question.

Let us further recall some standard notions. We denote by $\sigma
(.,.)$ the symplectic 2-form on the vector space $T^\ast(\R^n)$:
 \begin{equation}
   \symp{x}{\xi}{y}{\eta}= \sympl{x}{\xi}{y}{\eta},
 \end{equation}
and we denote by  $\{f,g\}$ the Poisson
bracket of two functions, 
\ie   
\begin{equation*}
  \{f,g\} = \sum_{j=1}^n \d_{\xi_j} f\ \d_{x_j} g -  
  \d_{x_j} f\ \d_{\xi_j} g. 
\end{equation*}
We shall use the notation $\#^w$ to denote the
composition of symbols in Weyl quantization,\ie, $a^w(x,D_x)\circ b  ^w(x,D_x) = (a\ \#^w b)^w   (x,D_x)$.
The following result is classical.
%%%%%%%%%%%%%%%%%%%%%%%%
% proposition          %
%%%%%%%%%%%%%%%%%%%%%%%%
\begin{proposition}
  \label{prop: Weyl composition}
  Let $a \in S^m$, $b \in S^{m'}$. Then  $a\ \#^w b \in S^{m+m'}$ and 
  \begin{align}
    \label{eq: composition formula Weyl}
    (a\ &\#^w b) (x,\xi) = \sum_{j=0}^k 
    \frac{1}{j!} \left(\frac{i}{2} \symp{D_x}{D_\xi}{D_y}{D_\eta} \right)^j \ 
    a(x,\xi) b(y,\eta) \left.\mstrut{0.4cm}\right|_{{y=x}\atop {\eta=\xi}}\\
    &+ \pi ^{-2n}\int\limits_0^1 \frac{(1-r)^k}{k!} 
    \mint{4} e^{i\Sigma(z,\zeta,t,\tau,\xi)} \left(\frac{i}{2} 
    \symp{D_x}{D_\zeta}{D_y}{D_\tau} \right)^{k+1}\!\!
    a(x+r z,\zeta) b(y+r t,\tau) \ d r \ d z\ d \zeta\ d t\ d \tau
    \left.\mstrut{0.4cm}\right|_{y=x},
    \nonumber
  \end{align}
  where $\Sigma(z,\zeta,t,\tau,\xi) = 2 (\sympl{t}{\tau-\xi}{z}{\zeta-\xi})$.
\end{proposition}
The result of Proposition~\ref{prop: Weyl composition} is to be
understood in the sense of oscillatory integrals (see \eg
\cite[Chapter 7.8]{Hoermander:V1}, \cite {AG:91}, \cite{GS:94} or
\cite{Kumano-go:81}).  For the sake of concision we have introduced
\begin{align*}
  \mint{n} := \underbrace{\int \cdots \int}_{n \text{ times}}, 
  \quad \text{for } n\geq 3,\ n \in \N.
\end{align*}
For the exposition to be self contained, we prove
Proposition~\ref{prop: Weyl composition} in Appendix~\ref{appendix:
  composition}.

We sometimes use the notion of multiple symbols. A function
$a(x,\xi,y,\eta) \in C^{\infty}(\R^{q_1}\times
\R^{p_1}\times\R^{q_2}\times \R^{p_2})$ is in $S^{m,m'}(\R^{q_1}\times
\R^{p_1}\times\R^{q_2}\times \R^{p_2})$, if for all multi-indices
$\alpha_1$, $\beta_1$, $\alpha_2$, $\beta_2$, there exists
$C_{\alpha_1 \alpha_2}^{\beta_1 \beta_2}>0$ such that
\begin{equation}
  |\partial_x^{\alpha_1} \partial_\xi^{\beta_1}
  \partial_y^{\alpha_2} \partial_\eta^{\beta_2} a(x,\xi,y,\eta)| 
  \leq C_{\alpha_1 \alpha_2}^{\beta_1 \beta_2}\: 
  \est{\xi}^{m - |\beta_1|}
  \est{\eta}^{m' -|\beta_2|},
\end{equation}
$x \in \R^{q_1}$, $y \in \R^{q_2}$, 
  $\xi \in \R^{p_1}$,  $\eta \in \R^{p_2}$ (see for instance \cite[Chapter 2]{Kumano-go:81}).

% We denote by $\sigma (.,.)$ the symplectic form on the vector space
% $T^\ast(\R^n)$:
% \begin{equation}
%   \sigma\big(({t},{\tau}),({z},{\zeta})\big) = \langle{t},{\tau}\rangle - \langle{z},{\zeta}\rangle.
% \end{equation}

  For $s \in \R$. We set $E^{(s)} := \est{D_x}^s = \Op(\est{\xi}^s)$,
  which realizes an isometry from $H^r(\R^n)$ onto $H^{r-s}(\R^n)$ for
  any $r \in \R$.  We denote by $\scp{.}{.}$ and $\|.\|$ the inner
  product and the norm of $L^2(\R^n)$, respectively and
  $\|\cdot\|_{H^s}$ for the norm on $H^s(\R^n)$, $s \in \R$.  For two
  Hilbert spaces $K$ and $L$, we use $\|\cdot\|_{(K,L)}$ to denote the
  norm in ${\mathcal L}(K,L)$, the set of bounded operators from $K$
  into $L$.

Our basic strategy is to obtain a bound for $\psi$DOs
involving a small parameter $h\geq 0$. In the following, we say that
an inequality holds uniformly in $h$ if it is the case when $h$ varies
in $[0,h_{\max}]$ for some $h_{\max}>0$.  In the sequel, $C$ will
denote a generic constant independent of $h$, whose value may change
from line to line.  
The semi-norms
\begin{equation}
  p_{\alpha \beta} (a) := \sup_{(x,\xi)\in \R^n\times \R^p}
  \est{\xi}^{-m + |\beta|} |\partial_x^\alpha
  \partial_\xi^\beta a (x,\xi)|
\end{equation}
endow a Fr\'{e}chet space structure to $S^m(\R^n\times \R^p)$.  In the
case of a symbol $a_h$ that depends on the parameter $h$ we shall say
that $a_h$ is in $S^m_{\rho,\delta}$ uniformly in $h$ if for all
$\al$, $\beta$ the semi-norm $p_{\al\beta}(a_h)$ is uniformly bounded
in $h$. Similarly, we shall say that an operator $A$ is in $\Psi^m$
uniformly in $h$ if its (Weyl) symbol is itself in $S^m$ uniformly in
$h$.

The outline of the article is as follows. Sections~\ref{sec: Hs bound}
and \ref{sec: convergence} are devoted to the multi-product representation
of solutions on $\R^n$. In Section~\ref{sec: Hs bound} we prove the
sharp Sobolev norm estimate (\ref{eq: intro sharp estimate}), which
leads in Section~\ref{sec: convergence} to the stability of the
multi-product representation.  We then prove convergence of the
multi-product representation in Section~\ref{sec: convergence}. Some
of the results of these two sections make use of composition-like
formulae, whose proofs are provided in Appendix~\ref{appendix:
  composition}. In Section~\ref{sec: manifold}, we address the
multi-product representation of solutions of a second-order
differential parabolic problem on a compact Riemannian manifold. As in
the previous sections we prove stability (in the $L^2$ case) through a
sharp operator norm estimate and we prove convergence of the
multi-product representation. The convergence proof requires an
analysis of the effect of a change of variables on symbols of the form
$e^{-h q(x,\xi)}$, from one local chart to another, which we
present in Appendix~\ref{sec: appendix change variables}.

\section{A sharp ${H^s}$ bound}
\label{sec: Hs bound}

We first make precise the assumption on the symbol $q(x,\xi)$
mentioned in the introduction.  
%%%%%%%%%%%%%%%%%%%%%%%%
% assumption           %
%%%%%%%%%%%%%%%%%%%%%%%%
\begin{assumption}
  \label{assumption: symbol q}
  The symbol $q$ is of the form
  $q=q_2 + q_1$, where $q_j \in S^j$,
  $j=1,2$, $q_2(x,\xi)$ is real-valued and for some $C\geq 0$ we have
  \begin{equation*}
    q_2(x,\xi) \geq C |\xi|^2, \quad x \in \R^n, \quad  \xi \in \R^n, \
    |\xi|\ \text{ sufficiently large.}
    \end{equation*}
\end{assumption}
Consequently,  for some $C\geq 0$, we have 
\begin{align}
  \label{eq: ellipticity}
  q_2(x,\xi) + \Re q_1(x,\xi) \geq C |\xi|^2,
  \quad x \in \R^n, \quad  \xi \in \R^n, \
  |\xi|\ \text{ sufficiently large, say}\ 
  |\xi|\geq\vartheta>0.
\end{align}
  As is stated in the introduction, our main aim is to deal with the operator 
  $p_h^{w}(x,D_x)$ where 
  \[p_h (x,\xi) = e^{-h q(x,\xi)}.\]
 It is well-known that
the $\psi$DO $p_h^{w}(x,D_x)$ is uniformly $H^s$-bounded in $h$, $s \in
\R$. Actually we have the following sharper estimate. 
%%%%%%%%%%%%%%%%%%%%%%%%
% Theorem 2.1            %
%%%%%%%%%%%%%%%%%%%%%%%%
\begin{theorem}
\label{theorem: sharp estimate}
Let $s \in\R$. There exists a constant $C \geq 0$ such that
\begin{align*}
  \| p_h^{w}(x,D_x)\|_{(H^s,H^s)} \leq 1 + C h,
\end{align*}
holds for all $h\geq 0$.
\end{theorem}

% \begin{remark}
%   In the case $q$ is polynomial, then if we write $q(x,\xi) =
%   q_2(x,\xi) + q_1(x,\xi) + q_0(x,\xi)$, with $q_j$ homogeneous of
%   degree $j$ in $\xi$, we have
%   \begin{align*}
%     h q(x,\xi) = \q(x,h^\hf \xi,h), \quad \text{with}\quad 
%     \q(x,\xi,h)
%     = q_2(x,\xi) + h^\hf q_1(x,\xi) + h q_0(x,\xi),
%   \end{align*}
%   and then obtain
%   \begin{align*}
%     p_h^{w}(x,D_x) 
%     &= (2\pi)^{-n}\iint e^{i\inp{x-y}{\xi}} e^{-\q((x+y)/2,h^\hf \xi,h)}\: 
%     u (y)\: d y \: d \xi\\
%     &= (2\pi h^\hf)^{-n}\iint e^{i\inp{x-y}{\xi}/h^\hf} e^{-\q((x+y)/2,\xi,h)}\: 
%     u (y)\: d y \: d \xi,
%   \end{align*}
%   and thus obtain a semi-classical formulation for $p_h^w(x,D_x)$,
%   with $h^\hf$ as the small  semi-classical parameter. We shall not exploit this point
%   of view here until Section~\ref{sec: manifold} in which we 
%   consider such operators on manifolds.
% \end{remark}
To prove Theorem 2.1 we shall need some
preliminary results.
%%%%%%%%%%%%%%%%%%%%%%%%
% Lemma 2.2            %
%%%%%%%%%%%%%%%%%%%%%%%%
\begin{lemma}
  \label{lemma: getting h out}
  {\rm (i)} Let $l \geq 0$ and $r \in S^l$. Then
  $h^{l/2} r\, p_h$ is in $S^0$ uniformly
  in $h$.

  \noindent
  {\rm (ii)} Let $\alpha$ and $\beta$ be multi-indices such that
  $|\alpha+\beta|\geq 1$.  Then, for any $0\leq m \leq 1$, we have
  $\partial_x^{\alpha} \partial_{\xi}^{\beta} p_h = h^m
  \tilde{p}_h^{m\alpha\beta}$, where $\tilde{p}_h^{m\alpha\beta}$ is
  in $S^{2m-|\beta|}$ uniformly in $h$.
\end{lemma}
\begin{proof}
  We have 
  \begin{equation*}
  (h\langle{\xi}^2\rangle)^j e^{- h \Re q(x,\xi)} \leq C_j,
 \quad x \in \R^n, \quad \xi\in \R^n,\quad h\geq 0,
\end{equation*}
  for all $j \in \N$ by (\ref{eq: ellipticity}), hence
  \begin{equation*}
    h^{l/2} |r(x,\xi) p_h(x,\xi)| \leq C, 
    \quad x \in \R^n, \quad \xi\in \R^n,\quad h\geq 0.
  \end{equation*}
  For multi-indices
  $\alpha$ and $\beta$  we observe that $h^{l/2} \partial_x^\alpha
  \partial_\xi^\beta (r(x,\xi) p_h(x,\xi))$ is a linear combination of terms
  of the form
  \begin{align*}
    h^{k+l/2} (\partial_x^{\alpha_0} \partial_\xi^{\beta_0} r) (x,\xi)
    (\partial_x^{\alpha_1} \partial_\xi^{\beta_1} q) (x,\xi)
    \cdots 
    (\partial_x^{\alpha_k} \partial_\xi^{\beta_k} q) (x,\xi) p_h(x,\xi),
  \end{align*}
  for $k\geq 0$, $\alpha_0+\alpha_1+\cdots + \alpha_k= \alpha$, 
  $\beta_0+\beta_1+\cdots
  + \beta_k= \beta$ and the absolute value of this term can be estimated by 
  \begin{align*}
    h^{k+l/2} \langle{\xi}\rangle^{l + 2k - |\beta|} |p_h(x,\xi)|
     \leq C \langle{\xi}\rangle^{-|\beta|}, 
     \quad x \in \R^n, \quad \xi\in \R^n,\quad h\geq 0,
  \end{align*}
  by (2.1), which concludes the proof of {\rm (i)}. For $|\alpha +
  \beta|\geq 1$, $\partial_x^\alpha \partial_\xi^\beta p_h(x,\xi)$ is
  a linear combination of terms of the form
  \begin{align*}
    h^{k} 
    (\partial_x^{\alpha_1} \partial_\xi^{\beta_1} q) (x,\xi)
    \cdots 
    (\partial_x^{\alpha_k} \partial_\xi^{\beta_k} q) (x,\xi) p_h(x,\xi),
  \end{align*}
  for $k\geq 1$, $\alpha_1+\cdots + \alpha_k= \alpha$, $\beta_1+\cdots
  + \beta_k= \beta$, which can be rewritten as $h^m\lambda_h(x,\xi)$, where 
  \begin{align*}
   \lambda_h(x,\xi) = (\partial_x^{\alpha_1} \partial_\xi^{\beta_1} q) (x,\xi) \cdots
    (\partial_x^{\alpha_k} \partial_\xi^{\beta_k} q) (x,\xi) \langle{\xi}\rangle^{2m -2k} 
    \left( h  \langle{\xi}\rangle^{2}\right)^{k-m} p_h(x,\xi).
  \end{align*}
Since $(\partial_x^{\alpha_1} \partial_\xi^{\beta_1} q) \cdots
    (\partial_x^{\alpha_k} \partial_\xi^{\beta_k} q) \in S^{2k - |\beta|}$, we see that $\lambda_h \in S^{2m - |\beta|}$ uniformly in $h$ by using {\rm (i)}.
\end{proof}
From  Weyl Calculus and the previous lemma we have the following
composition results for the symbol $p_h$.
%%%%%%%%%%%%%%%%%%%%%%%%
% proposition          %
%%%%%%%%%%%%%%%%%%%%%%%%
\begin{proposition}
  \label{prop:composition r ph}
  Let $r_h$ be bounded in $S^l$, $l \in \R$,  uniformly in $h$. We then have
  \begin{align}
    \label{eq: r p}
    r_h\ \#^w p_h = r_h p_h + h^{\hf} \lambda_h^{(0)} 
    = r_h p_h + h \lambda_h^{(1)}
    = r_h p_h + \frac{1}{2i} \{r_h, p_h\} + h \tilde{\lambda}^{(0)}_h,\\
    \label{eq: p r}
    p_h\ \#^w r_h = r_h p_h + h^{\hf} \mu_h^{(0)} 
    = r_h p_h + h \mu_h^{(1)}
    = r_h p_h + \frac{1}{2i} \{p_h,r_h\} + h \tilde{\mu}^{(0)}_h,
  \end{align}
  where $\lambda^{(0)}_h$, $\mu_h^{(0)}$, $\tilde{\lambda}^{(0)}_h$,
  and $\tilde{\mu}^{(0)}_h$ are in $S^{l}$ uniformly in $h$ and 
  $\lambda_h^{(1)}$ and $\mu_h^{(1)}$ are in $S^{l+1}$ uniformly in $h$.
\end{proposition}
To ease the reading of the article, the proof of
Proposition~\ref{prop:composition r ph} has been placed in
Appendix~\ref{appendix: composition}. We apply the result of
Proposition~\ref{prop:composition r ph} to prove the following lemma.
%%%%%%%%%%%%%%%%%%%%%%%%%%%%%%
% First Weyl calculus  lemma %
%%%%%%%%%%%%%%%%%%%%%%%%%%%%%%
\begin{lemma}
  \label{lemma: weyl calculus 1}
  We have $\overline{p_h} \ \#^w\, \est{\xi}^{2s}\,  \#^w  p_h 
  - \est{\xi}^{2s} |p_h|^2 = h k_h$ with $k_h$
  in $S^{2s}$ uniformly in $h$.
\end{lemma}
\begin{proof}
  By Proposition~\ref{prop:composition r ph} we have
  \begin{align*}
    \est{\xi}^{2s}\: \#^w p_h &=
    \est{\xi}^{2s}\:p_h 
    + \frac{1}{2i} \left\{ \est{\xi}^{2s}, p_h\right\} 
    + h \lambda_{1,h}
    =  \est{\xi}^{2s}\:p_h
    + \frac{1}{2i}\sum_{j=1}^n 
    (\d_{\xi_j}\est{\xi}^{2s})\ \d_{x_j} p_h   
    + h \lambda_{1,h},
  \end{align*}
  with $\lambda_{1,h}$ in $S^{2s}$ uniformly in $h$. We then obtain
  \begin{align*}
  \overline{p_h} \ \#^w\, \est{\xi}^{2s}\,  \#^w  p_h
  = \overline{p_h} \ \#^w \left(\mstrut{0.4cm}\right.\!\!
  \est{\xi}^{2s}\:p_h 
    + \frac{1}{2i}\sum_{j=1}^n 
    (\d_{\xi_j} \est{\xi}^{2s})\ \d_{x_j} p_h
    \!\!\left.\mstrut{0.4cm}\right) + h \lambda_{2,h},
  \end{align*}
  with $\lambda_{2,h}$ in $S^{2s}$ uniformly in $h$. 
  By
  Proposition~\ref{prop:composition r ph} we have
  \begin{align*}
     \overline{p_h} \ \#^w \left(\est{\xi}^{2s}\:p_h\right) 
     &= \est{\xi}^{2s} |p_h|^2 
     + \frac{1}{2i} \left\{
     \overline{p_h}, \est{\xi}^{2s} p_h 
     \right\} + h \lambda_{3,h}\\
     &= \est{\xi}^{2s} |p_h|^2 
     + \frac{1}{2i} \sum_{j=1}^n \left(
     (\d_{\xi_j} \overline{p_h}) \est{\xi}^{2s} \d_{x_j} p_h
     - (\d_{x_j} \overline{p_h})\ \d_{\xi_j} (\est{\xi}^{2s} p_h)\right)
    + h \lambda_{3,h},
  \end{align*}
  with $\lambda_{3,h}$ in $S^{2s}$ uniformly in $h$. 
  We also have
  \begin{align*}
     \overline{p_h} \ \#^w\frac{1}{2i}\sum_{j=1}^n 
     (\d_{\xi_j}\est{\xi}^{2s})\ \d_{x_j} p_h
    =\frac{1}{2i}\sum_{j=1}^n   
    \overline{p_h}\ 
    (\d_{\xi_j} \est{\xi}^{2s})  \d_{x_j} p_h + h \lambda_{4,h},
  \end{align*}
  with $\lambda_{4,h}$ in $S^{2s}$, uniformly in $h$, by
  Proposition~\ref{prop:composition r ph}. We have thus obtained
  \begin{align}
    \label{eq: simplified composition}
    \overline{p_h} \ \#^w\, \est{\xi}^{2s}\,  \#^w  p_h
    =\est{\xi}^{2s} |p_h|^2 + \frac{1}{2i} l_{h}^{(1)} 
    + \frac{1}{2i} \sum_{j=1}^n
    l_{h,j}^{(2)} + h \lambda_{5,h},
  \end{align}
  with $\lambda_{5,h}$ in $S^{2s}$ uniformly in $h$, 
  and 
  with
  \begin{align*}
     l_{h}^{(1)}  = \{\ovl{p_h},p_h\} \est{\xi}^{2s}
     =
     \sum_{j=1}^n (
     \d_{\xi_j} \ovl{p_h}\ \d_{x_j} p_h
     -\d_{x_j} \ovl{p_h}\ \d_{\xi_j}p_h)
     \est{\xi}^{2s}
  \end{align*}
  and
  \begin{align*}
    l_{h,j}^{(2)} =
    (\ovl{p_h}  \: \d_{x_j} p_h
    -p_h\ \d_{x_j} \ovl{p_h})
    \ \d_{\xi_j}\est{\xi}^{2s}.
  \end{align*}
  We introduce $\alpha := q_2 + \Re q_1$ and
  $\beta:= \Im q_1$. We have
  \begin{align}
    \label{eq: complex miracle}
     l_{h}^{(1)}  = 2i h^2 |p_h|^2 \est{\xi}^{2s} \ \{\al,\be\}
  \end{align}
  and 
  \begin{align*}
    l_{h,j}^{(2)}  =
    -2 i h |p_h|^2 (\d_{x_j}\beta) \ \d_{\xi_j}\est{\xi}^{2s}.
  \end{align*}
  Since $\al\in S^2$ and $\be \in S^1$ we then have
  $\{\al,\be\}\in S^2$. From Lemma~\ref{lemma: getting
  h out}, we thus obtain $l_{h}^{(1)} = h
  k_{h}^{(1)}$, with $k_{h}^{(1)}$ in $S^{2s}$
  uniformly in $h$. We also have that $l_{h,j}^{(2)} = h
  k_{h,j}^{(2)}$ with $k_{h,j}^{(2)}$ in $S^{2s}$
  uniformly in $h$, which from (\ref{eq: simplified composition}) concludes the proof.
\end{proof}
We shall also need the following lemma.
%%%%%%%%%%%%%%%%%%%%%%%%%%%%%%%
% Second Weyl calculus  lemma %
%%%%%%%%%%%%%%%%%%%%%%%%%%%%%%%
\begin{lemma}
  \label{lemma: weyl calculus 2}
  We have $\est{\xi}^{s}\, \#^w\, |p_h|^2\, \#^w\, \est{\xi}^{s} -
  \est{\xi}^{2s} |p_h|^2 = h k_h$ with $k_h$ in $S^{2s}$ uniformly in
  $h$.
\end{lemma}
%%%%%%%%%%%
% proof
%%%%%%%%%%%
\begin{proof}
  We set $\rho_h = \est{\xi}^{s}\,  \#^w\, |p_h|^2\,  \#^w\, \est{\xi}^{s}$.
  From Weyl calculus we have
  \begin{align*}
    \rho_h(x,\xi)
    = \pi^{-2n}\!\! \mint{4} e^{i\Sigma(z,t,\zeta,\tau,\xi)}
    \est{\zeta}^{s} \est{\tau}^{s}\ |p_h|^2(x+z+t,\xi)\  
    d z \ d \zeta \ d t \ d \tau
  \end{align*}
  where $\Sigma(z,t,\zeta,\tau,\xi) = 2(\inp{z}{\tau-\xi} - \inp{t}{\zeta-\xi})$.
  Arguing as in the proof of Proposition~\ref{prop: Weyl
  composition} in Appendix~\ref{appendix: composition} we write
  \begin{align}
    \rho_h (x,\xi) - &\est{\xi}^{2s} |p_h|^2 (x,\xi)= \pi^{-2n}\!\!
    \mint{4} e^{i\Sigma(z,t,\zeta,\tau,\xi)} \est{\zeta}^{s}
    \est{\tau}^{s}\ ( |p_h|^2(x+z+t,\xi) - |p_h|^2(x,\xi))\ 
    d z \ d \zeta \ d t \ d \tau\nonumber \\
    \label{eq: weyl calc intermediate term}
    &= \frac{i}{2} \pi^{-2n}\sum_{j=1}^{n}  \int\limits_0^1\!\! \mint{4}
    e^{i\Sigma(z,t,\zeta,\tau,\xi)} 
    (\d_{\tau_j} - \d_{\zeta_j}) (\est{\zeta}^{s}\est{\tau}^{s}) \:
    (\d_{x_j} |p_h|^2)(x+r(z+t),\xi)\ d r \
    d z \ d \zeta \ d t \ d \tau, 
  \end{align}
  by a first-order Taylor formula and integrations by
  parts \wrt $\zeta$ and $\tau$. Observing that we have  
  \begin{align*}
    0 &= \frac{i}{2} \sum_{j=1}^{n} (\d_{\tau_j} - \d_{\zeta_j})
    (\est{\zeta}^{s}\est{\tau}^{s}) \
    (\d_{x_j} |p_h|^2)(x,\xi) \left.\mstrut{0.3cm}\right|_{\zeta=\tau=\xi}\\
    &=\frac{i}{2} \pi^{-2n}\sum_{j=1}^{n}  \mint{4}
    e^{i\Sigma(z,t,\zeta,\tau,\xi)} 
    (\d_{\tau_j} - \d_{\zeta_j}) (\est{\zeta}^{s}\est{\tau}^{s}) \
    (\d_{x_j} |p_h|^2)(x,\xi)\ 
    d z \ d \zeta \ d t \ d \tau,
  \end{align*}
  we can proceed as in the proof of Proposition~\ref{prop: Weyl
  composition} (integration by parts \wrt $r$ in (\ref{eq: weyl calc
  intermediate term}) and further integrations by parts \wrt $\zeta$
  and $\tau$) and conclude after noting that $|p_h|^2$ satisfies the
  properties listed in Lemma~\ref{lemma: getting h out} like $p_h$.
\end{proof}
%%%%%%%%%%%%%%%%%%%%%%%%
% remark               %
%%%%%%%%%%%%%%%%%%%%%%%%
\begin{remark}
  \label{rem: jutification Weyl}
  Note that the use of the Weyl quantization is crucial in the proofs
  of Lemmata~\ref{lemma: weyl calculus 1} and \ref{lemma: weyl calculus
    2}. The use of the standard (left) quantization would only yield a
  result of the form $\overline{p_h} \ \#\, \est{\xi}^{2s}\,  \#\:  p_h - \est{\xi}^{2s} |p_h|^2 = h^\hf k_h$ with
  $k_h$ in $S^{2s}$ uniformly in $h$. Such a result would yield a $h^\hf$
  term in the statement of Theorem~\ref{theorem: sharp estimate} and
  the subsequent analysis would not carry through.
\end{remark}

We now define the symbol $\nu_h (x,\xi) = \frac{1 - |p_h|^2(x,\xi)}{h}$, for
$h >0$, and prove the following lemma.
%%%%%%%%%%%%%%%%%%%%%%%%
% lemma                %
%%%%%%%%%%%%%%%%%%%%%%%%
\begin{lemma}
  \label{lemma: symbol in S2}
  The symbol $\nu_h$ is in $S^2$ uniformly
  in $h$.
\end{lemma}
\begin{proof}
  We write $\nu_h (x,\xi) = 2 \Re q(x,\xi) \int_0^1 e^{-2 r h \Re
    q(x,\xi)}\: d r$. The integrand is in $S^0$ uniformly in $r \in
  [0,1]$ and $h$ by Lemma~\ref{lemma: getting h out} and $\Re q \in
  S^2$.
\end{proof}
%%%%%%%%%%%%%%%%%%%%%%%%
% Lemma                %
%%%%%%%%%%%%%%%%%%%%%%%%
\begin{lemma}
  \label{lemma: pre-sharp estimate}
The symbol $p_h$ is such that 
\begin{align*}
  \scp{(|p_h|^2)^w (x,D_x) u}{u} \leq (1 + C h)
  \|u\|^2,\quad u \in L^2(\R^n),
\end{align*}
for $C\geq 0$, uniformly in $h\geq 0$.
\end{lemma}
\begin{proof}
  Let $\chi \in \Con^\infty_c(\R^n)$, $0\leq \chi \leq 1$,  be such that $\chi(\xi)=1$ if
  $|\xi|\leq \vartheta$.
  Then we write
  \begin{align*}
    |p_h|^2(x,\xi) &= e^{-2h \Re q(x,\xi)}
    =e^{-2h \chi(\xi) \Re q(x,\xi)} e^{-2h (1-\chi(\xi)) \Re q(x,\xi)}
    = (1 + h \mu(x,\xi)) e^{-2h (1-\chi(\xi)) \Re q(x,\xi)},
  \end{align*}
  where $\mu(x,\xi) = - 2\chi(\xi) \Re q(x,\xi) \int_0^1 e^{-2r  h
    \chi(\xi) \Re q(x,\xi)} d r$. From \cite[18.1.10]{Hoermander:V3},
  the symbol $e^{-2r  h \chi(\xi) \Re q(x,\xi)}$ is in $S^0$ uniformly
  in $r$ and $h$; hence the symbol $\mu(x,\xi)$ is in $S^{0}$
  uniformly in $h$.  From \cite[Theorem 18.6.3]{Hoermander:V3} and
  Lemma~\ref{lemma: getting h out}, it is thus sufficient to prove the
  result with $p_h(x,\xi)$ replaced by $\tilde{p}_h(x,\xi) =e^{-h
    (1-\chi(\xi)) q(x,\xi)}$. We set
  \begin{align*}
    \tilde{\nu}_h (x,\xi) := \frac{1}{h}\left(1 - |\tilde{p}_h|^2(x,\xi)\right),
  \end{align*}
  for $h >0$. From (the proof of) Lemma~\ref{lemma: symbol in S2} we
  find that $\tilde{\nu}_h (x,\xi)$ is in $S^2$ uniformly in
  $h$. Since $1 - \chi(\xi) =0$ if $|\xi|\leq \vartheta$ and $\Re
  q(x,\xi) \geq 0$ if $|\xi|\geq \vartheta$, we observe that
  $\tilde{\nu}_h (x,\xi)\geq 0$. Then the Fefferman-Phong inequality
  reads (\cite{FP:78}, \cite[Corollary 18.6.11]{Hoermander:V3})
  \begin{align*}
    \left(\tilde{\nu}_h^w (x,D_x) u, u\right) \geq -C \norm{u}{L^2}^2, \quad u
    \in L^2(\R^n),
  \end{align*}
  for some non-negative constant $C$ that can be chosen uniformly in
  $h$. This yields
  \begin{align*}
    \norm{u}{L^2}^2  - \left((|\tilde{p}_h|^2)^w (x,D_x) u, u\right)\geq -C h  \norm{u}{L^2}^2, \quad u
    \in L^2(\R^n),
  \end{align*}
  which concludes the proof.
\end{proof}

We are now ready to prove Theorem~\ref{theorem: sharp
  estimate}.\\
%%%%%%%%%%%%%%%%%%%%%%%%
% Proof of theorem     %
%%%%%%%%%%%%%%%%%%%%%%%%
{\em {\bfseries Proof of Theorem~\ref{theorem: sharp estimate}.}} 
We use the following commutative diagram,
\[
\begin{CD}
  H^s @>{\quad p_h^w(x,D_x)\quad }>> H^s \\
    @V{E^{(s)}}VV   @VV{E^{(s)}}V \\
    L^2 @>{\qquad  T_h \qquad }>> L^2
\end{CD}
\]
and prove that the operator $T_h$ satisfies
$\| T_h\|_{(L^2,L^2)} \leq 1 + C h$.

The Weyl symbol of $T_h^\ast \circ T_h$ is given by 
\begin{align*}
  \sigma_h  = \est{\xi}^{-s}\, \#^w\: \ovl{p_h} \ \#^w\, \est{\xi}^{2s}\,
   \#^w p_h \ \#^w\, \est{\xi}^{-s}.
\end{align*}
By Lemmata~\ref{lemma: weyl calculus 1} and \ref{lemma: weyl calculus 2}, we have
$\sigma_h  = |p_h|^2 + h k_h$, with $k_h$ in $S^0$ uniformly in $h$.
 We note that $k_h(x,\xi)$ is real valued. To
estimate the $L^2$ operator norm of $T_h$ we write
\begin{align*}
  \|T_h u\|^2 
  = \scp{T_h^\ast \circ T_h u}{u}
 &=  \scp{(|p_h|^2)^w (x,D_x)  u}{u} 
  + h \scp{k_h^w(x,D_x) u }{u}\\
  &\leq \scp{(|p_h|^2)^w (x,D_x)  u}{u} + Ch \|u\|^2,
\end{align*}
for $C\geq 0$ \cite[Theorem 18.6.3]{Hoermander:V3}.  The result of
Theorem~\ref{theorem: sharp estimate} thus follows from
Lemma~\ref{lemma: pre-sharp estimate}. \hfill \qedsymbol

Let $m(x)$ be a smooth function that satisfies
  \begin{align}
    \label{eq: density property}
    0< m_{\min} \leq m(x) \leq m_{\max}< \infty,
  \end{align}
  along with all its derivatives.
  With such a function $m$, we define the following norm on $L^2(\R^n)$
  \begin{align*}
    \norm{f}{L^2(\R^n,m\, dx)}^2 = \int\limits_{\R^n} f^2(x)\, m(x) d x, 
  \end{align*}
  which is equivalent to the classical $L^2$ norm.  We shall need the
  following result in Section~\ref{sec: manifold}.
%%%%%%%%%%%%%%%%%%%%%%%%
% proposition          %
%%%%%%%%%%%%%%%%%%%%%%%%
\begin{proposition}
  \label{prop: sharp estimate density}
  There exists a constant $C \geq 0$ such that
  \begin{align*}
    \| p_h^{w}(x,D_x) u\|_{L^2(\R^n,m\, dx)} 
    \leq (1 + C h) \, \norm{u}{L^2(\R^n,m\, dx)},\qquad u \in L^2(\R^n),
  \end{align*}
  holds for all $h\geq 0$.
\end{proposition}
%%%% proof of proposition
\begin{proof}
  We follow the proof of Lemma~\ref{lemma: pre-sharp estimate} and use $\tilde{p}_h(x,\xi)$ in place of $p_h(x,\xi)$. We then set 
  \begin{align*}
    \tilde{\nu}_h (x,\xi) := \frac{m(x)}{h}
    \left(1 - |\tilde{p}_h|^2(x,\xi)\right),
  \end{align*}
  Then $\tilde{\nu}_h (x,\xi)\geq 0$ is in $S^2$ uniformly in $h$. The
  Fefferman-Phong inequality yields
  \begin{align*}
    \left(\tilde{\nu}_h^w (x,D_x) u, u\right) \geq -C \norm{u}{L^2}^2 
    \geq -C' \norm{u}{L^2(\R^n,m\, dx)}^2 , \quad u
    \in L^2(\R^n).
  \end{align*}
  This yields
  \begin{align*}
    \norm{u}{L^2(\R^n,m\, dx)}^2  
    - \left((m|\tilde{p}_h|^2)^w (x,D_x) u, u\right)
    \geq -C' h  \norm{u}{L^2(\R^n,m\, dx)}^2, \quad u
    \in L^2(\R^n),
  \end{align*}
  By Lemma~\ref{lemma: density conjugation} just below, we have
  \begin{align*}
    \left((m|\tilde{p}_h|^2)^w (x,D_x) u, u\right) 
    = \left(\ovl{\tilde{p}_h}^w(x,D_x) \circ m \circ (\tilde{p}_h^w(x,D_x) u 
      , u \right) + h (\lambda^w(x,D_x) u, u), 
  \end{align*}
  with $\lambda(x,\xi)$ in $S^0$ uniformly in $h$ and where $m$ stands
  for the associated multiplication operator here.  We
  conclude since $\ovl{\tilde{p}_h}^w(x,D_x) =
  (\tilde{p}_h^w(x,D_x))^\ast$.
\end{proof}
%%%%%%%%%%%%%%%%%%%%%%%%
% lemma                %
%%%%%%%%%%%%%%%%%%%%%%%%
\begin{lemma}
  \label{lemma: density conjugation}
  Let $f \in \Cinf(\R^n)$ be bounded along with all its derivatives. We have
  \begin{align*}
    \ovl{p_h}\, \#^w f\, \#^w p_h - f |p_h|^2 = h \lambda_h,
  \end{align*}
  with $\lambda_h$ in $S^0$ uniformly in $h$.
\end{lemma}
%%%% proof of lemma
\begin{proof}
  From Proposition~\ref{prop:composition r ph} we have
  \begin{align*}
    f \, \#^w p_h = f p_h - \frac{1}{2 i} 
    \sum_{j=1}^n (\d_{x_j} f)\ \d_{\xi_j} p_h 
    + h \lambda_{1,h},
  \end{align*}
  with $\lambda_{1,h}$ in $S^0$ uniformly in $h$. 
  By Proposition~\ref{prop:composition r ph} we also have
  \begin{align*}
    \ovl{p_h}\, \#^w (f p_h) &=  f |p_h|^2  + \frac{1}{2 i} \{\ovl{p_h},
    f p_h \} + h \lambda_{2,h}\\
    &= f |p_h|^2  + \frac{1}{2 i} \sum_{j=1}^n 
    \left( (\d_{\xi_j} \ovl{p_h})(\d_{x_j} f)  p_h 
    + (\d_{\xi_j} \ovl{p_h}) f (\d_{x_j} p_h) 
    - (\d_{x_j} \ovl{p_h})  f (\d_{\xi_j} p_h)\right) + h \lambda_{2,h},   
  \end{align*}
  with $\lambda_{2,h}$ in $S^0$ uniformly in $h$, and
  \begin{align*}
    \ovl{p_h}\, \#^w \left( (\d_{x_j} f) \d_{\xi_j} p_h\right) = 
    \ovl{p_h} (\d_{x_j} f) \d_{\xi_j} p_h + h\mu_{j,h},\quad j=1,\dots,n,
  \end{align*}
  with $\mu_{j,h}$ in $S^0$ uniformly in $h$.
  It follows that 
  \begin{align*}
    \ovl{p_h}\, \#^w f \, \#^w p_h = f |p_h|^2  
    - \frac{1}{2 i} \sum_{j=1}^n (\d_{x_j}f) (\ovl{p_h}\, \d_{\xi_j} p_h 
    - p_h\, \d_{\xi_j} \ovl{p_h})
    +\frac{1}{2 i} \, f \, \{\ovl{p_h}, p_h\}
    + h\lambda_{4,h},
  \end{align*}
  with $\lambda_{4,h}$ in $S^0$ uniformly in $h$. With the notation of
  the proof of Lemma~\ref{lemma: weyl calculus 1} (see expression
  (\ref{eq: complex miracle})) we have
  \begin{align*}
    \{\ovl{p_h}, p_h\} = 2i h^2 |p_h|^2\ \{\al,\beta\} =h
  k_{h}^{(1)},  
  \end{align*}
  with $k_{h}^{(1)}$ in $S^0$ uniformly in $h$ by Lemma~\ref{lemma: getting
  h out}, since $\beta$ is in
  $S^1$ and $\al \in S^2$.
  We also have 
  \begin{align*}
    \sum_{j=1}^n \d_{x_j}f(\ovl{p_h}\, \d_{\xi_j} p_h 
    -  p_h\d_{\xi_j} \ovl{p_h})  = h k_{h}^{(2)},
  \end{align*} 
  with $k_{h}^{(2)}$ in $S^0$ uniformly in $h$ by Lemma~\ref{lemma:
  getting h out}.
\end{proof}

\section{Multi-product representation: stability and convergence}
\label{sec: convergence}

We are interested in a representation of the solution operator
for the following parabolic Cauchy problem
\begin{align}
  \label{eq: cauchy pb 1}
  \d_t u + q^w(t,x,D_x) u &=0, \ \ \ 0< t\leq T,\\
  u \mid_{t=0} &= u_0 \in H^s(\R^n).
  \label{eq: cauchy pb 2}
\end{align} 
Here the symbol $q(t,x,\xi)$ is assumed to satisfy
Assumption~\ref{assumption: symbol q} uniformly \wrt the evolution
parameter $t$ and to remain in a bounded domain in $S^2$ as $t$
varies. We then note that the result of the previous section remains
valid in this case, \ie, the constant $C$ obtained in
Theorem~\ref{theorem: sharp estimate} is uniform \wrt $t$.  We denote
by $U(t',t)$ the solution operator to the evolution problem (\ref{eq:
  cauchy pb 1}).

Following \cite{LeRousseau:04}, we introduce the following
approximation of $U(t,0)$. With
$\P=\{t^{(0)},t^{(1)},\dots,t^{(N)}\}$, a subdivision of $[0,T]$ with
$0=t^{(0)}< t^{(1)}<\dots <t^{(N)}=T$, we define the following
multi-product
\begin{align}
  \label{eq: multi-product}
    \W_{\P,t} := \left\{
      \begin{array}{ll}
        P_{(t,0)} & \text{if }\ 0\leq t\leq t^{(1)},\\
        P_{(t,t^{(k)})} {\displaystyle \prod_{i=k}^{1}} 
	P_{(t^{(i)},t^{(i-1)})} 
        & \text{if }\ t^{(k)}\leq t\leq t^{(k+1)}.
      \end{array}
    \right.
  \end{align}
  where  $P_{(t',t)}$ is the $\psi$DO with Weyl symbol
  $p_{(t',t)}$ given by $p_{(t',t)}:= e^{-(t'-t)q(t,x,\xi)}$ for $t'\geq t$:
  \begin{align*}
    P_{(t',t)} v(x) = p_{(t',t)}^w(x,D_x) v(x) = 
    (2\pi)^{-n}\iint e^{i\inp{x-y}{\xi}} e^{-(t'-t)q(t,(x+y)/2,\xi)} 
  v (y)\: d y \: d \xi.
  \end{align*}
  We shall prove the convergence of $\W_{\P,t}$
  to $U(t,0)$ in some operator norms as well as its strong
  convergence.

\subsection{Stability}

As a consequence of the estimate proven in Theorem~\ref{theorem: sharp estimate}
we have the following proposition.
%%%%%%%%%%%%%%%%%%%%%%%%
% proposition          %
%%%%%%%%%%%%%%%%%%%%%%%%
\begin{proposition}
  \label{prop:Hs norm under control}
  Let $s \in \R$. There exists $K\geq 0$ such that for every subdivision
  $\P$ of $[0,T]$, we have
  \begin{align*}
    \forall t \in [0,T], \ \  \norm{\W_{\P,t}}{\sob{s}{s}} \leq e^{KT}.
  \end{align*}
\end{proposition}
\begin{proof}
  By Theorem~\ref{theorem: sharp estimate}, there exists $C\geq 0$ such
  that we have $\norm{P_{(t',t)}}{\sob{s}{s}}\leq 1+ C(t'-t)$ for
  all $t',t \in [0,T]$, $t'>t$; we then obtain
  \begin{equation*}
    \norm{\W_{\P,t}}{\sob{s}{s}}\leq \prod_{i=0}^{N-1} (1+C
    (t^{(i+1)}-t^{(i)} )).
  \end{equation*}
  Setting $U_\P = \ln
  \left(\prod_{i=0}^{N-1} (1+C(t^{(i+1)}-t^{(i )}))\right)$, we then have
  $U_\P \leq \sum_{i=0}^{N-1} C (t^{(i+1)}-t^{(i )}) = C T$. We thus obtain
  $\| \W_{\P,t}\|_{(H^{s},H^{s})} \leq e^{CT}$.
\end{proof}

\subsection{Convergence}
To obtain a convergence result we shall need the following assumption
on the regularity of the symbol $q(t,x,\xi)$ \wrt the evolution
parameter $t$.
%%%%%%%%%%%%%%%%%%%%%%%%
% assumption           %
%%%%%%%%%%%%%%%%%%%%%%%%
\begin{assumption}
  \label{assumption: regularity q}
  The symbol $q(t,x,\xi)$ is in $\Con^{0,\alpha}([0,T],S^2(\R^n\times \R^n))$, i.e.,
  H\"older continuous \wrt $t$ with values in $S^2$, in the sense
  that, for some $0< \alpha \leq 1$,
  \begin{align*}
    q(t',x,\xi) - q(t,x,\xi) = (t'-t)^\alpha\  \tilde{q}(t',t,x,\xi),\  \ 
    0 \leq t \leq t' \leq T,
  \end{align*}
  with $\tilde{q}(t',t,x,\xi)$ 
  in $S^2$ uniformly in $t'$ and $t$.
\end{assumption}

We now give some regularity properties for the approximation Ansatz
$\W_{\P,t}$ we have introduced.
%%%%%%%%%%%%%%%%%%%%%%%%
% lemma                %
%%%%%%%%%%%%%%%%%%%%%%%%
\begin{lemma}
  \label{lemma: t' -> pt't lipschitz}
  Let $s \in \R$ and $t'',t \in [0,T]$, with $t < t''$.  The map $t'
  \mapsto P_{(t',t)}$, for $t' \in [t,t'']$, is Lipschitz
  continuous with values in ${\mathcal L}(H^{s}(\R^n),H^{s-2}(\R^n))$. More
  precisely there exists $C>0$ such that, for all $v \in
  H^{s}(\R^n)$ and $t^{(1)}, t^{(2)} \in [t,t'']$,
  \begin{align*}
    \left\|\left(P_{(t^{(2)},t)}-P_{(t^{(1)},t)}\right)(v)\right\|_{H^{s-2}} 
    \leq C \left|t^{(2)}-t^{(1)}\right| \|v\|_{H^{s}}.
  \end{align*}
\end{lemma}
\begin{proof}
  We simply write  
  \begin{align*}
    (P_{(t^{(2)},t)}-P_{(t^{(1)},t)})(v) (x) =
    - (2 \pi)^{-n} \int\limits_{t^{(1)}}^{t^{(2)}}\!\!\iint e^{i \inp{x-y}{\xi}}
    e^{- (t'-t)q(t,(x+y)/2,\xi)}\
        q(t,(x+y)/2,\xi)\ \ v(y)\ d y\ d \xi \ d t'.
  \end{align*}
  We thus obtain a $\psi$DO whose Weyl symbol is in $S^2$ uniformly \wrt $t^{(2)}$ and $t^{(1)}$ and conclude with Theorem 18.6.3 in \cite{Hoermander:V3}.
\end{proof}
%%%%%%%%%%%%%%%%%%%%%%%%
% lemma                %
%%%%%%%%%%%%%%%%%%%%%%%%
\begin{lemma}
  \label{lemma: cont p, dt p}
  Let $s \in \R$, $t'',t \in [0,T]$, with $t <  t''$, and let $v
  \in H^{s}(\R^n)$.  Then the map $t'\mapsto P_{(t',t)}(v)$ is in
  $\Con^0([t,t''],H^{s}(\R^n)) \cap \Con^1([t,t''],H^{s-2}(\R^n))$.
\end{lemma}
\begin{proof}
  Let $t^{(1)} \in [t,t'']$ and let $\varepsilon>0$.  Choose $v_1 \in
  H^{s+2}(\R^n)$ such that $\| v - v_1\|_{H^{s}}\leq \varepsilon$.
  Then for $t^{(2)} \in [t,t'']$
  \begin{align}
    \left\| P_{(t^{(2)},t)}(v) - P_{(t^{(1)},t)}(v)\right\|_{H^{s}}
    &\leq 
    \left\|P_{(t^{(2)},t)}(v-v_1)\right\|_{H^{s}}
    + \left\|P_{(t^{(1)},t)}(v-v_1)\right\|_{H^{s}}
    + \left\|P_{(t^{(2)},t)}(v_1) - P_{(t^{(1)},t)}(v_1)\right\|_{H^{s}}
    \nonumber\\
    \label{estimate: Hs cont}
    &\leq 2 (1+ C (t''-t)) \varepsilon + C\left|t^{(2)}-t^{(1)}\right| \| v_1\|_{H^{s+2}}.
    \end{align}
    The continuity of the map follows. Differentiating
    $P_{(t',t)}(v)$ \wrt $t'$, we can prove that the resulting map
    $t' \mapsto \d_{t'}P_{(t',t)}(v)$ is Lipschitz continuous with
    values in ${\mathcal L}(H^{s+2}(\R^n),H^{s-2}(\R^n))$ following the proof of
    Lemma~\ref{lemma: t' -> pt't lipschitz}: there exists $C>0$ such
    that for all $w \in H^{s+2}(\R^n)$
    \begin{align*}
      \left\|(\d_{t'} P_{(t^{(2)},t)}
      -\d_{t'}P_{(t^{(1)},t)})(w)\right\|_{H^{s-2}} \leq C \left|t^{(2)}-t^{(1)}\right|
      \|w\|_{H^{s+2}}.
    \end{align*}
    Here $\d_{t'} P_{(t^{(i)},t)}$ means $\d_{t'} P_{(t',t)}
    \!\!\left.\mstrut{0.28cm}\right|_{t'=t^{(i)}}$.  We also see that
    the map $v \mapsto \d_{t'}P_{(t',t)}(v)$ is continuous from
    $H^{s}(\R^n)$ into $H^{s-2}(\R^n)$ with bounded continuity module:
    with $v \in H^{s}(\R^n)$, we make a similar choice as above for
    $v_1 \in H^{s+2}(\R^n)$ and obtain an estimate for
    \[ \left\| \d_{t'} P_{(t^{(2)},t)}(v) - \d_{t'}
    P_{(t^{(1)},t)}(v)\right\|_{H^{s-2}}
    \]
    of the same form as in (\ref{estimate: Hs cont}).
 \end{proof}
 Gathering the results of the previous lemmata we obtain the following
 regularity result for the Ansatz $\W_{\P,t}$.
%%%%%%%%%%%%%%%%%%%%%%%%
% proposition          %
%%%%%%%%%%%%%%%%%%%%%%%%
\begin{proposition} 
  Let $s \in \R$, let $u_0 \in H^{s}(\R^n)$.  Then the map $\W_{\P,t}
  (u_0)$ is in $\Con^0([0,T],H^{s}(\R^n))$ and piecewise
  $\Con^1([0,T],H^{s-2}(\R^n))$. 
\end{proposition}

The following energy estimate holds for a function $f(t)$ that is
in $\Con^0([0,T],H^{s}(\R^n))$ and piecewise
$\Con^1([0,T],H^{s-2}(\R^n))$ (by adapting the proof the energy estimate in
Section 6.5 in \cite{CP:82}):
\begin{align}
  \label{eq: energy estimate}
  \|f(t)\|_{H^{s-1}}^2 + \int\limits_0^T\|f(t')\|_{H^{s}}^2 \ d t'
  \leq C \left[\mstrut{0.5cm}\right.
    \|f(0)\|_{H^{s-1}}^2 + \int\limits_0^T \|\left(\d_{t'} + q^w(t',x,D_x) \right)f(t')\|_{H^{s-2}}^2 \ d t'
    \left.\mstrut{0.5cm}\right],
\end{align}
for all  $t \in [0,T]$. Once applied to $(U(t,0) - \W_{\P,t})(u_0)$ with $u_0 \in H^{s}(\R^n)$ we obtain
\begin{align}
  \|(U(t,0) - \W_{\P,t})(u_0)\|_{H^{s-1}}^2 
  &+ \int\limits_0^T\|(U(t,0) - \W_{\P,t})(u_0)\|_{H^{s}}^2 \ d t'
  \leq C  \int\limits_0^T \|\left(\d_{t'} + q^w(t',x,D_x) \right)\W_{\P,t'}(u_0)\|_{H^{s-2}}^2 \ d t' \nonumber\\
  &= C \sum_{j=0}^{N-1} \int\limits_{t^{(j)}}^{t^{(j+1)}} 
  \|\left(\d_{t'} + q^w(t',x,D_x) \right) P_{(t',t^{(j)})}\W_{\P,t^{(j)}}(u_0)\|_{H^{s-2}}^2 \ d t'   \nonumber\\
  \label{eq: applied energy estimate}
  &\leq C \sum_{j=0}^{N-1} \int\limits_{t^{(j)}}^{t^{(j+1)}} 
  \|\left(\d_{t'} + q^w(t',x,D_x) \right) P_{(t',t^{(j)})}\|_{\sob{s}{s-2}}^2\ d t'\
  e^{KT} \|u_0\|_{H^s}^2,
\end{align}
where we have used the stability result of Proposition~\ref{prop:Hs
  norm under control}.  It remains to estimate the Sobolev operator
norm of $\left(\d_{t'} + q^w(t',x,D_x) \right) P_{(t',t)}$, for $t'>t$,
which can be understood as estimating the {\em consistency} of the proposed
approximation Ansatz. This is the object of the following proposition.
%%%%%%%%%%%%%%%%%%%%%%%%
% proposition          %
%%%%%%%%%%%%%%%%%%%%%%%%
\begin{proposition}
  \label{prop: consistency error}
  Let $s \in \R$. There exists $C>0$ such that $\| (\d_{t'} +
  q^w(t',x,D_x)) P_{(t',t)}\|_{\sob{s}{s-2}} \leq C (t'-t)^\al$, for $0\leq t 
\leq t' \leq T$.
\end{proposition}
\begin{proof}
  We have
  \begin{align}
    \d_{t'} P_{(t',t)} u (x) 
    = -(2\pi)^{-n}\iint e^{i\inp{x-y}{\xi}} 
    q(t,(x+y)/2,\xi) e^{-(t'-t)q(t,(x+y)/2,\xi)}\: 
    u (y)\: d y \: d \xi,
  \end{align}
  and thus the operator $(\d_{t'} + q^w(t',x,D_x)) P_{(t',t)}$
  admits 
  \begin{equation*}
    \sigma_{(t',t)} = q(t',.,.)\ \#^w
  p_{(t',t)}-q(t,.,.) p_{(t',t)}
  \end{equation*}
  for its Weyl
  symbol. Since by Assumption~\ref{assumption: regularity q} we have
  \begin{align*}
    q(t',x,\xi) - q(t,x,\xi) = (t'-t)^\alpha\  \tilde{q}(t',t,x,\xi),\  \ 
    0 \leq t \leq t' \leq T,
  \end{align*}
  with $\tilde{q}(t',t,x,\xi)$ in $S^2$ uniformly in $t'$ and
  $t$. We can thus conclude with the following lemma.
\end{proof}
%%%%%%%%%%%%%%%%%%%%%%%%
% lemma                %
%%%%%%%%%%%%%%%%%%%%%%%%
\begin{lemma}
  \label{lemma: q comp p - qp}
  We have $q\ \#^w p_h - q p_h = h \lambda_h$ with $\lambda_h$ in
  $S^2$ uniformly in $h$.
\end{lemma}
\begin{proof}
  As in Section~\ref{sec: Hs bound}, we ignore the evolution parameter
  $t$ in the notation.  The result is however uniform \wrt $t$.
  By Proposition~\ref{prop:composition r ph} we have 
  \begin{align*}
    q\ \#^w p_h = q p_h + \frac{1}{2i} \{q, p_h\} + h \lambda_h,
  \end{align*}
  with $\lambda_h$ in $S^{2}$ uniformly in $h$. We note however that $\{q, p_h\}=0$.
\end{proof}
The result of Proposition~\ref{prop: consistency error} and estimate (\ref{eq: applied energy estimate}) yield
\begin{align}
  \label{eq: convergence estimation}
  \|(U(t,0) - \W_{\P,t})(u_0)\|_{H^{s-1}}
  + \left(\mstrut{0.45cm}\right. \!\! \int\limits_0^T\|(U(t,0) - \W_{\P,t})(u_0)\|_{H^{s}}^2 \ d t'\!\!\left.\mstrut{0.45cm}\right)^{\hf}
  \leq C T e^{CT}\: \Delta_\P^\al\: \|u_0\|_{H^s},
\end{align}
where $\Delta_\P = \max_{0\leq j \leq N-1} (t^{j+1} - t^{(j)})$. This
error estimate implies the following convergence results which
provides a representation of $U(t,0)$ by an infinite multi-product of
$\psi$DOs:  $U(t,0) =
\lim_{\Delta_\P\to 0} \W_{\P,t}$. We now state our main theorem.
%%%%%%%%%%%%%%%%%%%%%%%%
% theorem              %
%%%%%%%%%%%%%%%%%%%%%%%%
\begin{theorem}
  \label{theorem: main theorem Rn}
  Assume that $q(t,x,\xi)$ satisfies Assumptions~\ref{assumption:
    symbol q} and \ref{assumption: regularity q}. Then the
    approximation Ansatz $\W_{\P,t}$ converges to the solution
    operator $U(t,0)$ of the Cauchy problem~{\rm (\ref{eq: cauchy pb
    1})--(\ref{eq: cauchy pb 2})} in ${\mathcal L}(H^{s}(\R^n),H^{s-1+r}(\R^n))$
    uniformly \wrt $t$ as $\Delta_\P= \max_{0\leq j \leq N-1} (t^{j+1}
    - t^{(j)})$ goes to 0 with a convergence rate of order
    $\alpha(1-r)$:
  \begin{align*}  
    \| \W_{\P,t}  - U(t,0) \|_{(H^{s},H^{s-1+r})} 
    \leq C \Delta_\P^{\alpha(1-r)}, \quad  t\in [0,T], \quad 0\leq r<1.
  \end{align*}
  The operator $\W_{\P,t}$ also converges to $U(t,0)$ in
  $L^2(0,T,{\mathcal L}(H^s(\R^n),H^s(\R^n)))$ with a convergence rate of order
  $\alpha$:
  \begin{equation*} 
    \left(\mstrut{0.4cm}\right.\!
    \int_0^T \| \W_{\P,t}  - U(t,0) \|_{(H^s,H^s)}^2 d t
    \!\!\left.\mstrut{0.4cm}\right)^\hf \leq C \Delta_\P^{\alpha}.
  \end{equation*}
Furthermore $\W_{\P,t}$ strongly converges to $U(t,0)$
  in ${\mathcal L}(H^{s}(\R^n),H^{s}(\R^n))$ uniformly \wrt $t \in
  [0,T]$.
\end{theorem}
\begin{proof}
  The first two results are consequences of (\ref{eq: convergence
  estimation}).  The proof of the first result for $r \neq 0$ follows
  by interpolation between Sobolev Spaces \cite{LM:68}.

  Let $u_0 \in H^{s}(\R^n)$ and let $\varepsilon >0$. For the strong
  convergence in $H^{s}(\R^n)$ we choose $u_1 \in H^{s+1}(\R^n)$ such
  that $\|u_0 - u_1\|_{H^{s}}\leq \varepsilon$. We then write
  \begin{multline*}
    \|\W_{\P,t}(u_0) - U(t,0)(u_0)\|_{H^s} \leq
    \|\W_{\P,t}(u_0-u_1)\|_{H^s}
    + \|\W_{\P,t}(u_1) - U(t,0)(u_1)\|_{H^s}\\
    + \|U(t,0)(u_0-u_1)\|_{H^s}
    \leq C\varepsilon + C\Delta_\P^\alpha  \|u_1\|_{H^{s+1}},
  \end{multline*}
  from the case $r=0$ of the first part of the theorem and from the
  stability of $\W_{\P,t}$ (Proposition~\ref{prop:Hs norm under control}). This
  last estimate is uniform \wrt $t \in [0,T]$ and yields the result.
\end{proof}

\section{Multi-product representation on a compact manifold}
\label{sec: manifold}
\subsection{Notation and setting}
\label{secM: intro}
We shall now consider the case of a parabolic equation on an
$n$-dimensional compact $\Cinf$-Riemannian manifold $(\M,g)$, where
$g$ is a smooth Riemannian metric.  We let $A$ be a second-order
elliptic differential operator on $\M$ whose principal part, $A_2$, is
given by the Laplace-Beltrami operator on $\M$, which reads
\begin{align*}
  A_2 = - {\mathsf g}^{-\hf} \d_i \left( {\mathsf g}^{\hf} g^{ij} \d_j\right), 
\end{align*}
in local coordinates, where ${\mathsf g} = \det (g_{ij})$.  Other
uniformly-elliptic operators can be considered by changing the metric.
We choose here to focus on the differential case instead of the
pseudodifferential case because the full symbol of the operator can
then be completely defined on the manifold $\M$.

We allow the operator $A$  to depend on an evolution
parameter $t$. We shall thus assume that the metric is itself time-dependent, yet continuous \wrt $t$, $g=g(t,x)$, and satisfies
\begin{equation}\
  \label{eqM: property metric}
  0<c\leq g(t,x) \leq C< \infty, \qquad t \in [0,T], \quad x \in \M.
\end{equation}
For the $L^2$ norm on $\M$, we shall use the metric $g(0,x)$ as a reference
metric.  We set ${\mathsf g}_0(x)= {\mathsf g}(0,x)$.  We then denote
by $d v$ the volume form which is given by $d v = {\mathsf g}_0^{1/2}(x)\, d
x$ in local coordinates. The $L^2$-inner product is then given by
$\scp{u}{w} = \int_\M u\, \ovl{w}\, d v$ \cite{hebey:96}.

Since we are going to consider an infinite product of $\psi$DOs, a
little attention should be paid to a finite atlas. We shall use an
atlas $\A=(\theta_i, \psi_i)_{i\in \I}$ of $\M$, $|\I|< \infty$, with
$\psi_i: \theta_i \to \tilde{\theta}_i$, where $\tilde{\theta}_i$ is a
smooth bounded open subset of $\R^n$. For $i \in \I$, we set
\begin{align*}
  \J_i := \{ j \in \I;\ \theta_i \cap \theta_j \neq \emptyset\},\quad 
  \J_i^{(2)} := \{ l \in \J_j;\ j \in \J_i\}, 
\end{align*}
which lists the neighboring charts and the ``second''-neighboring
charts for the chart $(\theta_i, \psi_i)$.  For technical reasons, we
shall assume that there exists a coarser finite atlas $\B=(\Theta_k,
\Psi_k)_{k\in \K}$ of $\M$, $\Psi_k: \Theta_k \to \tilde{\Theta}_k
\subset \R^n$, such that for each chart $(\theta_i, \psi_i)\in \A$
there exists a chart $(\Theta_{k(i)}, \Psi_{k(i)})\in \B$, such that
\begin{equation*}
\bigcup_{l \in \J_i^{(2)}} \theta_l \Subset
\Theta_{k(i)},
\end{equation*}
\ie, $\Theta_{k(i)}$ contains all the ``second''-neighbors of
$\theta_i$. This is always possible by choosing the atlas $\A$
sufficiently fine. We shall denote by $a_i(t)$, $i\in \I$, the Weyl symbol of
$A(t)$ in each local chart $(\theta_i,\psi_i)$.

We set $(\varphi_i)_{i\in \I}$ as a family of $\Cinf$ real-valued
functions defined on $\M$ such that the functions $(\varphi_i^2)_{i\in
  \I}$ form a partition of unity subordinated to the open covering
$(\theta_i)_{i\in \I}$, \ie,
\begin{align*}
  \supp (\varphi_i) \subset \theta_i, \quad 0\leq \varphi_i \leq
  1,\quad i \in \I,\quad \text{and}\quad \sum_{i\in \I} \varphi_i^2 =1.
\end{align*}
We denote
\begin{equation*}
  \tilde{\varphi}_i = (\psi_i^{-1})^\ast \varphi_i = \varphi_i \circ \psi_i^{-1},
\end{equation*}
and similarly, for $l \in \J^{(2)}_i$, we shall set  
\begin{equation*}
  \hat{\varphi}_l = (\Psi_{k(i)}^{-1})^\ast \varphi_l,
\end{equation*}
with $\Psi_{k(i)}$ as above, when there is no possible confusion on $k(i)$.

We set $Q(t)$ as the elliptic operator on $\M$ defined through
\begin{align*}
  A(t) = \sum_{i\in \I} \varphi_i \circ Q(t) \circ \varphi_i.
\end{align*}
The construction of $Q$ can be done recursively: we write $Q=Q_2 + Q_1
+ Q_0$, with $Q_l$ a differential operator of order $l$, $l=0,1,2$ and
obtain
\begin{align*}
  Q_2 = A,\quad  \ Q_1 = - \sum_{i\in \I} [\varphi_i, Q_2] \circ \varphi_i, 
  \quad 
  Q_0 = - \sum_{i\in \I} [\varphi_i, Q_1] \circ \varphi_i. 
\end{align*}
The recursion stops after two iterations since we consider
differential operators here.

In each local chart $(\theta_i,\psi_i)$, $i \in \I$, we denote by $q_i(t,x,\xi)$ the Weyl symbol of $Q(t)$, \ie,
\begin{align*}
  \forall u \in \Cinfc(\theta_i), \quad Q(t) u  = \psi_i^\ast
  \left(q_i^w(t,x,D_x) ((\psi_i^{-1})^\ast u)\right),
\end{align*}
or equivalently
\begin{align*}
  \forall \tilde{u} \in \Cinfc(\tilde{\theta}_i),
  \quad q_i^w(t,x,D_x) \tilde{u} = (\psi_i^{-1})^\ast
  \left(Q(t) (\psi_i^\ast \tilde{u})\right).
\end{align*}
The symbol $q_i(t,x,\xi)$ is uniquely defined since $Q(t)$ is a
differential operator. We also let $\hat{q}_k(t,x,\xi)$ be the Weyl
symbol of $Q(t)$ in the chart $(\Theta_{k}, \Psi_{k})$, $k \in \K$. From~(\ref{eqM:
  property metric}) we then have
%%%%%%%%%%%%%%%%%%%%%%%%
% lemma                %
%%%%%%%%%%%%%%%%%%%%%%%%
\begin{lemma}
  \label{lemmaM: property q}
  In each chart the symbol of $Q(t)$ satisfies the properties of
  Assumption~\ref{assumption: symbol q}.
\end{lemma}

We set
\begin{align*}
    p_{i,(t',t)}(x,\xi) = e^{-(t'-t) q_i(t,x,\xi)}, 
    \quad i \in \I, \quad 0\leq t\leq t'\leq T,\quad  x \in \tilde{\theta}_i,\ \xi \in \R^n.
\end{align*}
With these symbols in $S^0(\tilde{\theta}_i\times \R^n)$, we define the
following  $\psi$DOs on $\M$:
\begin{align}
  \label{eqM: localized propagator}
  P_{i,(t',t)} u &:= \varphi_i \circ \psi_i^\ast \circ 
  p_{i,(t',t)}^w(x,D_x) \circ (\psi_i^{-1})^\ast \circ \varphi_i
  =  \psi_i^\ast \circ  (\tilde{\varphi}_i \circ 
  p_{i,(t',t)}^w(x,D_x) \circ \tilde{\varphi}_i) \circ (\psi_i^{-1})^\ast,
  \\
  \label{eqM: propagator}
  P_{(t',t)} &:= \sum_{i\in \I} P_{i,(t',t)},
\end{align}
where $\varphi_i$ and $\tilde{\varphi}_i$ are understood here as
multiplication operators.    The operator
$P_{(t',t)}$ is the counterpart of the operator $p^w_{(t',t)}(x,D_x)$
introduced in Sections~\ref{sec: Hs bound} and \ref{sec: convergence}.
We shall compose such operators in the form of a multi-product as is
done in Section~\ref{sec: convergence} to obtain a representation of
the solution operator to the following well-posed parabolic Cauchy
problem on $\M$
\begin{align}
  \label{eqM: cauchy pb 1}
  \d_t u + A(t) u &=0, \ \ \ 0< t\leq T,\\
  u \mid_{t=0} &= u_0 \in H^s(\M).
  \label{eqM: cauchy pb 2}
\end{align}
We denote by $U(t',t)$ the solution operator of (\ref{eqM: cauchy pb
1})--(\ref{eqM: cauchy pb 2}) and we define the multi-product operator
$\W_{\P,t}$ as in (\ref{eq: multi-product}) for a subdivision $\P=\{t^{(0)},t^{(1)},\dots,t^{(N)}\}$ of
$[0,T]$:
\begin{align}
  \label{eqM: multi-product}
    \W_{\P,t} := \left\{
      \begin{array}{ll}
        P_{(t,0)} & \text{if }\ 0\leq t\leq t^{(1)},\\
        P_{(t,t^{(k)})} {\displaystyle \prod_{i=k}^{1}} 
	P_{(t^{(i)},t^{(i-1)})} 
        & \text{if }\ t^{(k)}\leq t\leq t^{(k+1)}.
      \end{array}
    \right.
  \end{align}

We shall make the following regularity assumption on the operator $A(t)$,
which is equivalent to  that made in  Section~\ref{sec: convergence}
(Assumption~\ref{assumption: regularity q}).
\begin{assumption}
  \label{assumptionM: regularity q}
  The symbol of $A(t)$ is H\"older continuous of order $\al$,
  $0<\al\leq 1$, \wrt $t$ with values in $S^2$: for each chart
  $(\theta_i, \psi_i)$ we have $a_i \in
  \Con^{0,\al}([0,T],S^2(\R^n\times \R^n))$, in the sense that,
  \begin{align*}
    a_i(t',x,\xi) - a_i(t,x,\xi) = (t'-t)^\al\  \tilde{a}_i(t',t,x,\xi),\  \ 
    0 \leq t \leq t' \leq T,
  \end{align*}
  with $\tilde{a}_i(t',t,x,\xi)$ in $S^2$ uniformly in $t'$ and
  $t$. Note that the same property then holds for the symbol of $A(t)$
  in any chart.
\end{assumption}
This property naturally translates
to the symbols $q_i(t)$, $i\in \I$.

\begin{remark}
  The form we have chosen for the operator $P_{(t',t)}$ can be
  motivated at this point. First, a natural requirement is that
  $P_{(t,t)} = \id$, which is achieved since $\sum_{i\in \I}
  \varphi_i^2 =1$. Second, the consistency analysis of
  Proposition~\ref{prop: consistency error} gears towards having
  $\left. \left(\d_t' P_{(t',t)} - A(t')\circ P_{(t',t)}\right)
  \right|_{t'=t} =0$, which is achieved here thanks to the form we
  have chosen for the differential operator $Q(t)$.
\end{remark}
As in Section~\ref{sec: convergence}, we first need to address the
stability of the multi-product. Here, we shall only consider  the
$L^2$ case.

\subsection{$\bld{L^2}$ Stability}
As in Section~\ref{sec: Hs bound}, we find a sharp estimate of the
$L^2$-norm of the operator $P_{(t',t)}$ over $\M$. 
%%%%%%%%%%%%%%%%%%%%%%%%
% Theorem              %
%%%%%%%%%%%%%%%%%%%%%%%%
\begin{theorem}
\label{theoremM: sharp estimate}
There exists a constant $C \geq 0$ such that
\begin{align*}
  \| P_{(t',t)}\|_{(L^2(\M),L^2(\M))} \leq 1 + C (t'-t),
\end{align*}
holds for all $0\leq t \leq t' \leq T$.
\end{theorem}
Therefore, as in Section~\ref{sec: convergence}, we obtain the
following stability result for $\W_{\P,t}$.
%%%%%%%%%%%%%%%%%%%%%%%%
% corollary            %
%%%%%%%%%%%%%%%%%%%%%%%%
\begin{corollary}
  \label{cor: stability}
  There exists $K\geq 0$ such that for every subdivision
  $\P$ of $[0,T]$, we have
  \begin{align*}
    \forall t \in [0,T], \ \  \norm{\W_{\P,t}}{(L^2(\M),L^2(\M))} \leq e^{KT}.
  \end{align*}
\end{corollary}
%%% proof 
\begin{proof}[{\bfseries{Proof of Theorem~\ref{theoremM: sharp estimate}.}}]
  We let $u$, $w \in L^2(\M)$. We have
  \begin{align*}
    \scp{P_{(t'-t)}u}{w} 
    = \sum_{i\in \I} \int\limits_\M \varphi_i \ovl{w}\,  \psi_i^\ast ( 
  p_{i,(t',t)}^w(x,D_x) ((\psi_i^{-1})^\ast ( \varphi_i u)))\, d v
  = \sum_{i\in \I} \int\limits_{\tilde{\theta}_i} 
  (\ovl{\tilde{\varphi}_i \tilde{w}_i})
  \, p_{i,(t',t)}^w(x,D_x) (\tilde{\varphi}_i \tilde{u}_i)\ 
     {\mathsf g}_0^{1/2}(x)\, d x,
  \end{align*}
  where $\tilde{u}_i$ (\resp $\tilde{w}_i$) is the pullback of
  $u|_{\theta_i}$ (\resp $w|_{\theta_i}$) by $\psi_i^{-1}$.  We now
  extend the symbol $q_i(t,.)$ to $\R^n\times \R^n$ to obtain a symbol
  satisfying Assumption~\ref{assumption: symbol q} like its
  counterpart in Section~\ref{sec: Hs bound}. We still denote by
  $q_i(t,.)$ this extended symbol.  Then, by Proposition~\ref{prop:
    sharp estimate density}, for all $i \in \I$, there exists $C_i\geq
  0$ such that
  \begin{align*}
    \norm{p_{i,(t',t)}^w(x,D_x)}
    {(L^2(\R^n, {\mathsf g}_0^{1/2}\, d x),L^2(\R^n, {\mathsf g}_0^{1/2}\, d x))} 
    \leq 1 + C_i(t'-t),
  \end{align*}
  where ${\mathsf g}_0$ is also extended from $\tilde{\theta}_i$ to
  $\R^n$, yet still preserving Property (\ref{eq: density property}).
  With $C= \max_{i\in \I} C_i$ (recall that $\I$ is finite) we thus
  obtain
  \begin{align*}
    |\scp{P_{(t'-t)}u}{w}|
    \leq   (1+C(t'-t)) \sum_{i\in \I} 
    \norm{\tilde{\varphi}_i \tilde{w}_i}
    {L^2(\tilde{\theta}_i, {\mathsf g}_0^{1/2}\, d x)}\,
    \norm{\tilde{\varphi}_i \tilde{u}_i}
    {L^2(\tilde{\theta}_i, {\mathsf g}_0^{1/2}\, d x)}.
  \end{align*}
  A Cauchy-Schwarz inequality then yields
  \begin{align*}
    |\scp{P_{(t'-t)}u}{w}|
    \leq (1+C(t'-t)) 
    \left(\mstrut{0.4cm}\right.\!\!
    \sum_{i\in \I} \norm{\tilde{\varphi}_i \tilde{w}_i}
    {L^2(\tilde{\theta}_i, {\mathsf g}_0^{1/2}\, d x)}^2
    \!\!\left.\mstrut{0.4cm}\right)^{\hf}
    \left(\mstrut{0.4cm}\right.\!\!
    \sum_{i\in \I} \norm{\tilde{\varphi}_i \tilde{u}_i}
    {L^2(\tilde{\theta}_i, {\mathsf g}_0^{1/2}\, d x)}^2
    \!\!\left.\mstrut{0.4cm}\right)^{\hf}.
  \end{align*}
  Observing that 
  \begin{align*}
    \sum_{i\in \I} \norm{\tilde{\varphi}_i \tilde{u}_i}
    {L^2(\tilde{\theta}_i, {\mathsf g}_0^{1/2}\, d x)}^2
    = \sum_{i\in \I} \int\limits_{\tilde{\theta}_i} \tilde{\varphi}_i^2 \tilde{u}_i^2
    \,{\mathsf g}_0^{1/2}(x)\, d x 
    = \sum_{i\in \I}  \int\limits_{\M} \varphi_i^2 u^2\,  d v
    = \norm{u}{L^2(\M)}^2, 
  \end{align*}
  since $\sum_{i\in \I}   \varphi_i^2 =1$, we find
  \begin{align*}
    |\scp{P_{(t'-t)}u}{w}|
    \leq (1+C(t'-t)) \,
    \norm{w}{L^2(\M)}\, \norm{u}{L^2(\M)},
  \end{align*}
  which concludes the proof.
\end{proof}
\subsection{Consistency estimate}

As in Section~\ref{sec: convergence}, Proposition~\ref{prop:
  consistency error}, for the case of $\R^n$, we shall now analyze the symbol 
of the operator $(\d_{t'} + A(t')) P_{(t',t)}$ and prove the following proposition that corresponds to a consistency estimate. 
%%%%%%%%%%%%%%%%%%%%%%%%
% proposition          %
%%%%%%%%%%%%%%%%%%%%%%%%
\begin{proposition}
  \label{prop: consistency manifold}
  Let $0\leq t \leq t'\leq T$. We have 
  \begin{align*}
    (\d_{t'} + A(t')) \circ P_{(t',t)} = (t'-t)^\al L_{(t',t)},
    \qquad \text{with}\quad L_{(t',t)} \in \Psi^2(\M),
  \end{align*}
  and for all $s \in \R$, there exists $C\geq 0$ such that  
  \begin{align}
    \label{eqM: norm estimate consistency}
    \norm{L_{(t',t)}}{(H^s(\M),H^{s-2}(\M))} \leq C,
  \end{align}
  uniformly in $t'$ and $t$. 
\end{proposition}
%%%%%%%%%%%%%%%%%%%%%%%%
% figure               %
%%%%%%%%%%%%%%%%%%%%%%%%
\begin{figure}
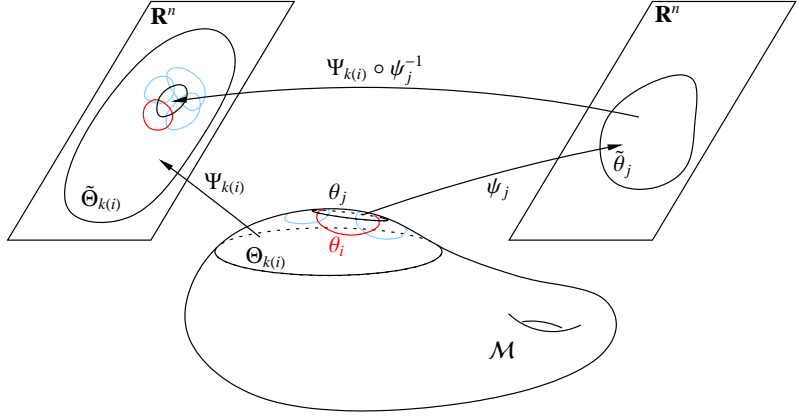

% pstex_t file
\input charts2.pstex_t
\caption{Change of variables bringing the analysis to the chart
  $(\Theta_{k(i)}, \Psi_{k(i)})$ for the charts $(\theta_i, \psi_i)$
  and $(\theta_j, \psi_j)$, $j \in \J_i$, and their neighboring charts.}
\label{fig: change of variable}
\end{figure} 
%%%% proof of proposition
\begin{proof}
  For $u \in \Cinf(\M)$ we have $ u =\sum_{i\in \I} \varphi_i^2 u$. 
  It thus suffices to take $u_i \in \Cinf(\M)$, with $\supp(u_i)\subset
  \theta_i$, for some $i \in \I$, and to prove that we have
  \begin{align*}
    (\d_{t'} + A(t')) P_{(t',t)} (u_i) 
    = (t'-t) L_{i,(t',t)} (u_i),\qquad L_{i,(t',t)}, \in \Psi^2(\M),
  \end{align*}
  and that $L_{i,(t',t)}$ satisfies (\ref{eqM: norm estimate
    consistency}) uniformly in $t'$ and $t$. 

  For concision we write $\hat{q}$ for $\hat{q}_{k(i)}$ here. Let us
  recall that $\hat{q}_{k}$ is the Weyl symbol of $Q(t)$ in the chart
  $(\Theta_k, \Psi_k)$, $k \in K$.
  We set $\hat{p}_{(t',t)}(x,\xi) := e^{-(t'-t) \hat{q}(t,
    x,\xi)}$. Making use of the assumption made on the chart
  $(\Theta_{k(i)}, \Psi_{k(i)})$, we consider the action of the change
  of variables $\kappa =\Psi_{k(i)} \circ \psi_j^{-1}$ on the
  operators $\tilde{\varphi}_j \circ p_{j,(t',t)}^w(x,D_x) \circ
  \tilde{\varphi}_j \in \Psi^0(\tilde{\theta}_j)$ for $j \in \J_i$
  (see Figure~\ref{fig: change of variable}). By Lemma~\ref{lemma:
    change of variable chi ph chi}, we obtain
  \begin{align*}
    %\label{eqM: form of the operator after change of varibales.}
     P_{(t',t)} u_i
     =\Psi_{k(i)}^\ast \circ \left(\mstrut{0.45cm}\right. 
     \sum_{j\in \J_i}  \hat{\varphi}_j
    \circ \hat{p}_{(t',t)}^w(x,D_x)  \circ \hat{\varphi}_j
    \!\!\left.\mstrut{0.45cm}\right) 
  \circ (\Psi_{k(i)}^{-1})^\ast u_i + (t'-t) R^{(0)}_{(t',t)} u_i,
  \end{align*}
  with $R^{(0)}_{(t',t)}$ in $\Psi^0(\M)$ uniformly in $t'$ and $t$. 
  We then have 
  \begin{align}
     \label{eqM: A P}
     A(t') \circ P_{(t',t)} u_i
     &= \sum_{j\in \J_i} \sum_{l \in \J_i^{(2)}}  
     \Psi_{k(i)}^\ast\circ \left(\mstrut{0.45cm}\right. 
     \hat{\varphi}_l \circ  
       \hat{q}^w (t',x,D_x) \circ \hat{\varphi}_l
       \hat{\varphi}_j \circ 
       \hat{p}_{(t',t)}^w(x,D_x) \circ \hat{\varphi}_j
       \left. \mstrut{0.45cm}\right) 
     \circ (\Psi_{k(i)}^{-1})^\ast u_i\\ 
     &+ (t'-t)R^{(1)}_{(t',t)}  u_i, \nonumber
   \end{align}  
   where $R^{(1)}_{(t',t)}$ is in $\Psi^2(\M)$ uniformly in $t'$ and
   $t$. 

   From~(\ref{eqM: localized propagator}) we have
   \begin{align*}
     \d_{t'} P_{j,(t',t)} u_i 
     =  \psi_j^\ast \circ  (\tilde{\varphi}_j \circ 
     (q_j(t,.) p_{j,(t',t)})^w(x,D_x) \circ \tilde{\varphi}_j) 
     \circ (\psi_j^{-1})^\ast u_i,\qquad j \in \J_i,
   \end{align*}
   which we may write
    \begin{align*}
     \d_{t'} P_{j,(t',t)} u_i 
     =  &\psi_j^\ast \circ  (\tilde{\varphi}_j \circ 
     q_j^w(t,x,D_x) \circ p_{j,(t',t)}^w(x,D_x) \circ \tilde{\varphi}_j) 
     \circ (\psi_j^{-1})^\ast u_i\\
     &+ (t'-t) \psi_j^\ast  
     (\tilde{\varphi}_j \circ  \tilde{R}_{j,(t'-t)}^{(2)}
     \circ \tilde{\varphi}_j) 
     \circ (\psi_j^{-1})^\ast u_i, \qquad  j \in \J_i,
   \end{align*}
   where $ \tilde{R}^{(2)}_{j,(t',t)}$ is in
   $\Psi^2(\tilde{\theta}_j)$ uniformly in $t'$ and $t$ by
   Lemma~\ref{lemma: q comp p - qp}. We choose $\chi_j \in
   \Cinfc(\tilde{\theta}_j)$ such that $\chi_j$ is equal to one on
   $\supp(\tilde{\varphi}_j)$. We then have 
   \begin{align*}
     \d_{t'} P_{j,(t',t)} u_i 
     =  &\psi_j^\ast \circ  (\tilde{\varphi}_j \circ 
     q_j^w(t,x,D_x) \circ \chi_j \circ p_{j,(t',t)}^w(x,D_x) \circ \chi_j 
     \circ \tilde{\varphi}_j) 
     \circ (\psi_j^{-1})^\ast u_i\\
     &+ (t'-t) \psi_j^\ast  
     (\tilde{\varphi}_j \circ  \tilde{R}_{j,(t'-t)}^{(2)}
     \circ \tilde{\varphi}_j) 
     \circ (\psi_j^{-1})^\ast u_i, \qquad  j \in \J_i,
   \end{align*}
   recalling that $q_j^w(t,x,D_x)$ is a differential operator, hence a
   local operator. Applying Lemma~\ref{lemma: change of variable
     chi ph chi} to $\chi_j \circ p_{j,(t',t)}^w(x,D_x) \circ \chi_j$
   we obtain
   \begin{align*}
     \d_{t'} P_{j,(t',t)} u_i 
     =  &\Psi_{k(i)}^\ast \circ  (\hat{\varphi}_j \circ 
     \hat{q}^w(t,x,D_x) \circ \hat{\chi}_j \circ \hat{p}_{(t',t)}^w(x,D_x) 
     \circ \hat{\chi}_j 
     \circ \hat{\varphi}_j) 
     \circ (\Psi_{k(i)}^{-1})^\ast u_i+ (t'-t) R_{j,(t'-t)}^{(3)}
      u_i, \quad  j \in \J_i,
   \end{align*}
   where $ R^{(3)}_{j,(t',t)}$ is in $\Psi^2(\M)$ uniformly in $t'$
   and $t$ and $\hat{\chi_j} = \left( \psi_j \circ \Psi_{k(i)}^{-1}\right)^\ast \chi_j$.  Using again that $\hat{q}^w(t,x,D_x)$ is a differential
   operator, we finally obtain
   \begin{align}
     \label{eqM: dt P}
     \d_{t'} P_{(t',t)} u_i 
     =  &\sum_{j\in \J_i}\Psi_{k(i)}^\ast \circ  (\hat{\varphi}_j \circ 
     \hat{q}^w(t,x,D_x) \circ \hat{p}_{(t',t)}^w(x,D_x)  
     \circ \hat{\varphi}_j) 
     \circ (\Psi_{k(i)}^{-1})^\ast u_i
     + (t'-t) R_{(t'-t)}^{(3)}
      u_i,
   \end{align}
   where  $ R^{(3)}_{(t',t)}$ is in
   $\Psi^2(\M)$ uniformly in $t'$ and $t$.

   The operators $R^{(1)}_{(t',t)}$ in (\ref{eqM: A P}) and
   $R^{(3)}_{(t',t)}$ in (\ref{eqM: dt P}) will contribute to the
   operator $L_{i,(t',t)}$ and we discard them from the subsequent
   analysis. Observe that we may change the sums over $j \in \J_i$ to
   sums over $j\in \J_i^{(2)}$  (\ref{eqM: A
     P}) and in (\ref{eqM: dt P}) since we only consider the action of the two operators on
   $u_i$.

   Now that we have brought the analysis to the open set
   $\tilde{\Theta}_{k(i)}$, we shall consider and analyze the following
   symbol, $\sigma_{(t',t)}$, which corresponds to the operator
   $(\Psi_{k(i)}^{-1})^\ast \circ (\d_t +A(t')) \circ P_{(t',t)} \circ
   \Psi_{k(i)}^\ast$ ignoring the operators $R^{(1)}_{(t',t)}$ and
   $R^{(3)}_{(t',t)}$ as explained above:
  \begin{align*}
    \sigma_{(t',t)}(x,\xi) = - \sum_{j\in \J_i^{(2)}} \hat{\varphi}_j\, \#^w
    \hat{q}(t,.)\, \#^w \hat{p}_{(t',t)}\, \#^w  \hat{\varphi}_j
    +\underbrace{\sum_{j,l\in \J_i^{(2)}}  
    \hat{\varphi}_l \, \#^w
       \hat{q} (t',.)\, \#^w \hat{\varphi}_l\, 
       \hat{\varphi}_j \, \#^w
       \hat{p}_{(t',t)}\, \#^w \hat{\varphi}_j}_{=:\sigma^{(1)}_{(t',t)}},
  \end{align*}
  keeping in mind that we only consider the action of the associated
  operator on $(\Psi_{k(i)}^{-1})^\ast u_i$ whose support is compact
  and contained in $\Psi_{k(i)}(\theta_i)$. We extend the symbol
  $\hat{q}(t,.)$ to $\R^n\times \R^n$ to obtain a symbol satisfying
  Assumption~\ref{assumption: symbol q} like its counterpart in
  Section~\ref{sec: Hs bound}.  We still denote by $\hat{q}(t,.)$ this
  extended symbol. We may then use global symbols in $\R^n$.  As in
  the proof of Proposition~\ref{prop: consistency error}, we may
  replace $\hat{q}(t,.)$ by $\hat{q}(t',.)$ by
  Assumption~\ref{assumptionM: regularity q}. For the symbol
  $\sigma_{(t',t)}(x,\xi)$, this yields an error term of the form
  $(t'-t)^\al \lambda^{(0)}_{(t',t)}(x,\xi)$, with $\lambda^{(0)}_{(t',t)}(x,\xi)$ in
  $S^2$ uniformly in $t'$ and $t$, that will contribute to the
  operator $L_{i,(t',t)}$. We thus discard this term in the subsequent
  analysis and we still denote by $\sigma_{(t',t)}(x,\xi)$ the
  modified symbol.

  We now set $\chi = 1 - \sum_{j\in \J_i^{(2)}} \hat{\varphi}_j^2$
  and we  write
  \begin{align*}
    \sigma^{(1)}_{(t',t)} &= 
    \sum_{j\in \J_i^{(2)}} \hat{\varphi}_j \, \#^w
    \hat{q} (t',.)\, \#^w
    \hat{\varphi}_j^2 \, \#^w
    \hat{p}_{(t',t)}\, \#^w \hat{\varphi}_j
    +
    \sum_{{\mbox{\scriptsize $j,l \in \J_i^{(2)}$}}
      \atop{\mbox{\scriptsize $l\neq j$}}}
    \hat{\varphi}_l \, \#^w
    \hat{q} (t',.)\, \#^w \hat{\varphi}_l\, 
    \hat{\varphi}_j \, \#^w
    \hat{p}_{(t',t)}\, \#^w \hat{\varphi}_j\\
    &= \sum_{j\in \J_i^{(2)}} \hat{\varphi}_j \, \#^w
    \hat{q} (t',.)\, \#^w
    (1-\chi) \, \#^w
    \hat{p}_{(t',t)}\, \#^w \hat{\varphi}_j
    +
    \sum_{{\mbox{\scriptsize $l,j \in \J_i^{(2)}$}}
      \atop{\mbox{\scriptsize $l\neq j$}}}
    \hat{\varphi}_l \, \#^w
    \hat{q} (t',.)\, \#^w \hat{\varphi}_l\, 
    \hat{\varphi}_j \, \#^w
    \hat{p}_{(t',t)}\, \#^w \hat{\varphi}_j\\
    &- \sum_{{\mbox{\scriptsize $l,j \in \J_i^{(2)}$}}
      \atop{\mbox{\scriptsize $l\neq j$}}}
    \hat{\varphi}_j \, \#^w
    \hat{q} (t',.)\, \#^w \hat{\varphi}_l^2\, 
    \, \#^w
    \hat{p}_{(t',t)}\, \#^w \hat{\varphi}_j,
  \end{align*}
  which yields
  \begin{align}
    \label{eqM: symbol sigma}
    \sigma_{(t',t)}(x,\xi) 
    &= - \sum_{j\in \J_i^{(2)}} \hat{\varphi}_j\, \#^w
    \hat{q}(t,.)\, \#^w \chi \# \hat{p}_{(t',t)}\, \#^w  \hat{\varphi}_j\\
    &+
    \sum_{{\mbox{\scriptsize $l,j \in \J_i^{(2)}$}}
      \atop{\mbox{\scriptsize $l\neq j$}}}
    \hat{\varphi}_l \, \#^w
    \hat{q} (t',.)\, \#^w 
    \left(\mstrut{0.40cm}\right. 
    \underbrace{\hat{\varphi}_j \hat{\varphi}_l \, \#^w
    \hat{p}_{(t',t)}\, \#^w \hat{\varphi}_j
    -  \hat{\varphi}_j^2 \, \#^w
    \hat{p}_{(t',t)}\, \#^w \hat{\varphi}_l}_{=:\sigma^{(2)}_{(t',t)}(x,\xi)}
    \left.\mstrut{0.40cm}\right)
  \end{align}
  From Weyl calculus \cite{Hoermander:V3}, since $\supp(\,
  \chi)\cap \supp((\Psi_{k(i)}^{-1})^\ast u_i) = \emptyset$, we find
  that the first term in the \rhs of (\ref{eqM: symbol sigma}) can be
  written in the form $(t'-t) \lambda^{(1)} (x,\xi)$, with
  $\lambda^{(1)}$ in $S^2$ uniformly in $t'$ and $t$, making use of
  the composition formula~(\ref{eq: composition formula Weyl}) and
  Lemma~\ref{lemma: getting h out}.
  
  Applying Proposition~\ref{prop: amplitude to Weyl symbol} (with
  $k=1$), we find
  \begin{align*}
    \sigma^{(2)}_{(t',t)}(x,\xi)
    =\sum_{m=1}^n \frac{i}{2}\left( 
      2 (\d_{x_m}\hat{\varphi}_l ) \hat{\varphi}_j^2 
      - \hat{\varphi}_l \d_{x_m}(\hat{\varphi}_j^2)
    \right)\d_{\xi_m} \hat{p}_{(t',t)}
    + (t'-t) \lambda^{(2)}(x,\xi),
  \end{align*}
  with $\lambda^{(2)}$ in $S^0$ uniformly in $t'$ and $t$ arguing as
  in the proof of Proposition~\ref{prop:composition r ph} in
  Appendix~\ref{appendix: composition} (using Lemma~\ref{lemma:
    getting h out} and Theorem 2.2.5 in \cite{Kumano-go:81}). 
  Therefore, we are now left with computing
  \begin{align*}
    \sigma^{(3)}_{(t',t)}(x,\xi)
    &= 
    \frac{i}{2} \sum_{m=1}^n  \sum_{{\mbox{\scriptsize $l,j \in \J_i^{(2)}$}}
      \atop{\mbox{\scriptsize $l\neq j$}}}
    \hat{\varphi}_l \, \#^w
    \hat{q} (t',.)\, \#^w \left[
      \left( 
        2 (\d_{x_m}\hat{\varphi}_l ) \hat{\varphi}_j^2 
        - \hat{\varphi}_l \d_{x_m}(\hat{\varphi}_j^2)
      \right)\d_{\xi_m} \hat{p}_{(t',t)}
    \right]\\
    &=\frac{i}{2} \sum_{m=1}^n  \sum_{{\mbox{\scriptsize $l,j \in \J_i^{(2)}$}}
      \atop{\mbox{\scriptsize $l\neq j$}}}
    \left( 
       (\d_{x_m}(\hat{\varphi}_l^2 )) \hat{\varphi}_j^2 
        - \hat{\varphi}_l^2 \d_{x_m}(\hat{\varphi}_j^2)
      \right)\hat{q}(t',x,\xi) \d_{\xi_m} \hat{p}_{(t',t)}
      + (t'-t) \lambda^{(3)}(x,\xi),
  \end{align*}
  with $\lambda^{(3)}$ in $S^2$ uniformly in $t'$ and $t$ by the 
  composition formula~(\ref{eq: composition formula Weyl}) and
  Lemma~\ref{lemma: getting h out}. Observing that the first term just
  obtained in fact vanishes, we finally have
  $\sigma_{(t',t)}(x,\xi)=(t'-t) \lambda(x,\xi)$ with $\lambda(x,\xi)$
  in $S^2$ uniformly in $t'$ and $t$. This concludes the proof.
\end{proof}

\subsection{Convergence and representation theorem}

 We observe that the energy estimate~(\ref{eq: energy estimate}) also
holds for the differential operator $A(t)$ on $\M$ (since the proof
relies on the G{\aa}rding inequality which holds for positive elliptic
operators on $\M$). Combined with the ($L^2$) stability result of
Corollary~\ref{cor: stability} and the consistency estimate of
Proposition~\ref{prop: consistency manifold}, the energy estimate
yields, as in Section~\ref{sec: convergence}, the following
representation theorem through the convergence of $\W_{\P,t}$ to
$U(t,0)$, the solution operator of the parabolic Cauchy problem
(\ref{eqM: cauchy pb 1})--(\ref{eqM: cauchy pb 2}): $U(t,0) =
\lim_{\Delta_\P\to 0} \W_{\P,t}$ in the following sense. 
%%%%%%%%%%%%%%%%%%%%%%%%
% theorem              %
%%%%%%%%%%%%%%%%%%%%%%%%
\begin{theorem}
  \label{theorem: main theorem Manifold}
  Assume that $A(t)$ satisfies Assumption~\ref{assumptionM: regularity
  q}. Then the approximation Ansatz $\W_{\P,t}$ converges to the
  solution operator $U(t,0)$ of the Cauchy problem~{\rm (\ref{eqM: cauchy
  pb 1})--(\ref{eqM: cauchy pb 2})} in ${\mathcal L}(L^2(\M),H^{-1+r}(\M))$
  uniformly \wrt $t$ as $\Delta_\P = \max_{0\leq j \leq N-1} (t^{j+1}
  - t^{(j)}) $ goes to 0 with a convergence rate of order
  $\alpha(1-r)$:
  \begin{align*}  
    \| \W_{\P,t}  - U(t,0) \|_{(L^2,H^{-1+r})} 
    \leq C \Delta_\P^{\alpha(1-r)}, \quad  t\in [0,T], \quad 0\leq r<1.
  \end{align*}
  The operator $\W_{\P,t}$ also converges to $U(t,0)$ in
  $L^2(0,T,{\mathcal L}(L^2(\M),L^2(\M)))$ with a convergence rate of order
  $\alpha$:
  \begin{align*} 
    \left(\mstrut{0.4cm}\right.\!
    \int_0^T \| \W_{\P,t}  - U(t,0) \|_{(L^2,L^2)}^2 d t
    \!\!\left.\mstrut{0.4cm}\right)^\hf \leq C \Delta_\P^{\alpha}.
  \end{align*}
  Furthermore $\W_{\P,t}$ strongly converges to $U(t,0)$
  in ${\mathcal L}(L^2(\M),L^2(\M))$ uniformly \wrt $t \in
  [0,T]$.
\end{theorem}

%appendix
\appendix
\section{Proofs of composition-like formulae}
\label{appendix: composition}

We prove Proposition~\ref{prop: Weyl composition} and derive
composition results for the symbol $p_h (x,\xi) = e^{-h q(x,\xi)}$.
\subsection{Proof of Proposition~\ref{prop: Weyl composition}}
From Weyl Calculus we have
\begin{align*}
    (a\ &\#^w b) (x,\xi) =
    \pi ^{-2n} 
    \mint{4} e^{i\Sigma(z,\zeta,t,\tau,\xi)} 
    a(x+ z,\zeta) b(x+ t,\tau) \ d z\ d \zeta\ d t\ d \tau,
  \end{align*}
where $\Sigma(z,\zeta,t,\tau,\xi) = 2
(\sympl{t}{\tau-\xi}{z}{\zeta-\xi})$ (see \cite{Hoermander:V3}, p.~152).
This  yields
\begin{align*}
    (a\ &\#^w b) (x,\xi) - (a b) (x,\xi)
    =
    \pi ^{-2n} 
    \mint{4} e^{i\Sigma(z,\zeta,t,\tau,\xi)} 
    \left(\mstrut{0.3cm} 
    a(x+ z,\zeta) b(x+ t,\tau) - a(x,\zeta) b(x,\tau)\right) 
    \ d z\ d \zeta\ d t\ d \tau\\
    &=\sum_{j=1}^n \pi ^{-2n} \int\limits_0^1
    \mint{4} e^{i\Sigma(z,\zeta,t,\tau,\xi)} 
    \left(\mstrut{0.3cm} 
    z_j \d_{x_j} a(x+ r z,\zeta) b(x+ r t,\tau) 
    + t_j a(x+ r z,\zeta)  \d_{x_j} b(x+ r t,\tau) \right) 
    \ d r \ d z\ d \zeta\ d t\ d \tau
  \end{align*}
by a first-order Taylor formula.  In the first (\resp second) term
  that we have obtained, we write
  \begin{align*}
    &z_j  e^{i\Sigma(z,\zeta,t,\tau,\xi)}
    =- \frac{i}{2}   \d_{\tau_j} e^{i\Sigma(z,\zeta,t,\tau,\xi)}
    \ \quad
    (\text{\resp}\ \  
    t_j  e^{i\Sigma(z,\zeta,t,\tau,\xi)}
    = \frac{i}{2} \d_{\zeta_j} e^{i\Sigma(z,\zeta,t,\tau,\xi)}).
  \end{align*}
Integration by parts \wrt $\tau$ and $\zeta$ in the oscillatory
integral yields
\begin{align*}
    (a\ &\#^w b) (x,\xi) - (a b) (x,\xi)
    =
    \sum_{j=1}^n \frac{\pi ^{-2n}}{2i}
    \int\limits_0^1
    \mint{4} e^{i\Sigma(z,\zeta,t,\tau,\xi)} 
    \left(\mstrut{0.3cm} 
     \d_{\xi_j} a(x+ r z,\zeta)  \d_{x_j} b(x+ r t,\tau) \right.\\
     &\hspace*{6.5cm}\left. \mstrut{0.3cm} 
     - \d_{x_j} a(x+ r z,\zeta) \d_{\xi_j} b(x+ r t,\tau) \right) 
    \ d r \ d z\ d \zeta\ d t\ d \tau,\\
    &= \frac{i\pi ^{-2n}}{2}
    \int\limits_0^1
    \mint{4} e^{i\Sigma(z,\zeta,t,\tau,\xi)} 
    \symp{D_x}{D_\zeta}{D_y}{D_\tau}
    \left( \mstrut{0.3cm} 
     a(x+ r z,\zeta) b(y+ r t,\tau) \right) 
    \ d r \ d z\ d \zeta\ d t\ d \tau\left.\mstrut{0.4cm}\right|_{y=x},
  \end{align*}
which gives the result of Proposition~\ref{prop: Weyl composition} for
$k=0$.To proceed further we integrate by parts \wrt $r$ and
obtain
\begin{align*}
    (a\ &\#^w b) (x,\xi) - (a b) (x,\xi)
    = 
    \frac{i\pi ^{-2n}}{2} \mint{4} e^{i\Sigma(z,\zeta,t,\tau,\xi)} 
    \symp{D_x}{D_\zeta}{D_y}{D_\tau}
    \left( \mstrut{0.3cm} 
     a(x,\zeta) b(y,\tau) \right) \ d z\ d \zeta\ d t\ d \tau
     \left.\mstrut{0.4cm}\right|_{y=x}
     \\
    &+\sum_{j=1}^n  \frac{i\pi ^{-2n}}{2}
    \int\limits_0^1 (1-r) 
    \mint{4} e^{i\Sigma(z,\zeta,t,\tau,\xi)} 
    \symp{D_x}{D_\zeta}{D_y}{D_\tau}
    \left( \mstrut{0.3cm}
    z_j \d_{x_j} a(x+ r z,\zeta) b(y+ r t,\tau) \right.\\
    &\hspace*{7.5cm}
    \left.+ t_j a(x+ r z,\zeta)  \d_{x_j} b(y+ r t,\tau)
     \mstrut{0.3cm} \right) 
    \ d r \ d z\ d \zeta\ d t\ d \tau
    \left.\mstrut{0.4cm}\right|_{y=x}\\
    &= \frac{i}{2}  
    \symp{D_x}{D_\xi}{D_y}{D_\eta}
    \left.\left( \mstrut{0.3cm} 
     a(x,\xi) b(y,\eta) \right)\right|_{{y=x} \atop {\eta=\xi}}\\
     &+ \pi ^{-2n} \left(\frac{i}{2}\right)^2
     \int\limits_0^1 (1-r) 
    \mint{4} e^{i\Sigma(z,\zeta,t,\tau,\xi)} 
    \left(\symp{D_x}{D_\zeta}{D_y}{D_\tau}\right)^2 
    \left( \mstrut{0.3cm} 
     a(x+ r z,\zeta) b(y+ r t,\tau) \right) 
    \ d r \ d z\ d \zeta\ d t\ d \tau
    \left.\mstrut{0.4cm}\right|_{y=x},
  \end{align*}
which gives the result for $k=1$. Formula~\ref{eq:
composition formula Weyl} then follows from induction
by integration par parts \wrt $r$ each time. \hfill \qedsymbol

\subsection{From amplitudes to symbols}
Here we give a formula of the form of (\ref{eq: composition formula
  Weyl}) to compute the Weyl symbol of a $\psi$DO
starting from an arbitrary amplitude.
%%%%%%%%%%%%%%%%%%%%%%%%
% proposition          %
%%%%%%%%%%%%%%%%%%%%%%%%
\begin{proposition}
  \label{prop: amplitude to Weyl symbol}
  Let $a(x,y,\xi) \in S^m(\R^n\times\R^n\times \R^n)$ be the amplitude
  of a $\psi$DO $A$, \ie, 
  \begin{equation*}
    A u(x) = (2\pi)^{-n}\iint e^{i\inp{x-y}{\xi}} a(x,y,\xi)\: u (y)\: d y \: d \xi.
  \end{equation*}
  The Weyl symbol $b$ of $A$, \ie $A= b^w(x,D_x)$,
  is then given by
  \begin{align*}
  %\label{eq: amplitude to Weyl symbol}
  b(x,\xi) &= \left.e^{\frac{i}{2} (\inp{D_y}{D_\xi}- \inp{D_x}{D_\xi})}
    a(x,y,\xi)\right|_{y=x}
    = \pi^{-n} \iint e^{2i \inp{z}{\zeta-\xi}}  a(x+z,x-z,\zeta)
    \, d z\ d \zeta\\
    &= \sum_{j=0}^k  \frac{1}{j!}  \left(\frac{i}{2} \inp{\d_x-\d_y}{\d_\xi} \right)^j \
    \left. a(x,y,\xi)\mstrut{0.4cm}\right|_{y=x}\nonumber\\
    &\hspace*{.1cm}+\pi^{-n}
    \int\limits_0^1\frac{(1-r)^k}{k!}\iint e^{2i \inp{z}{\zeta-\xi}}
    \left(\frac{i}{2} \inp{\d_x-\d_y}{\d_\xi} \right)^{k+1}
    a(x+r z,y-r z,\zeta)\ d r \ d z\ d \zeta \!\left.\mstrut{0.4cm}\right|_{y=x}.
  \nonumber
\end{align*}
\end{proposition}
The proof is analogous to that of
Proposition~\ref{prop: Weyl composition} given above.

\subsection{Proof of Proposition~\ref{prop:composition r ph}}

We prove the results for $r_h \ \#^w p_h$. The results for $p_h\
  \#^w r_h$ follow similarly. We first use Proposition~\ref{prop: Weyl
  composition} for $k=0$:
  \begin{align}
    \label{eq: comp 2}
    r_h\ &\#^w p_h (x,\xi) =  
    r_h(x,\xi) p_h(x,\xi) 
    + \frac{\pi ^{-2n}}{2i}\sum_{j=1}^n \int\limits_0^1
    \mint{4} e^{i\Sigma(z,\zeta,t,\tau,\xi)}  
    \left( \mstrut{0.3cm}\right.
    \d_{\xi_j} r_h(x+r z,\zeta)\, \d_{x_j}p_h(x+r t,\tau)\\
    &\hspace*{6.5cm}- \d_{x_j} r_h(x+r z,\zeta)\, \d_{\xi_j}p_h(x+r t,\tau)
    \left. \mstrut{0.3cm} \right)
    \ d r \ d z\ d \zeta\ d t\ d \tau.\nonumber
  \end{align}
By Lemma~\ref{lemma: getting h out}, we have
\begin{align*}
  \d_{x_j} p_h = h \nu_{h (j)}, \quad 
  \d_{\xi_j} p_h = h \nu_{h}^{(j)},
\end{align*}
with $ \nu_{h (j)}$ in $S^2$ and $\nu_{h}^{(j)}$ in $S^1$
uniformly in $h$. We thus observe that the last term in (\ref{eq: comp
2}) can be written as a linear combination of
  terms of the form
\begin{align}
  \label{eq: comp 3}
 h \int\limits_0^1 \mint{4} e^{i\Sigma(z,\zeta,t,\tau,\xi)}  
    \nu_{1,h}(x+r z,\zeta) \nu_{2,h}(x+r t,\tau)
    \ d r \ d z\ d \zeta\ d t\ d \tau,
\end{align}
where $\nu_{1,h}$ and $\nu_{2,h}$ are respectively in $S^{m_1}$ and
$S^{m_2}$ uniformly in $h$ with $m_1 + m_2 =l+1$.
Setting
\begin{align*}
  \nu_h (x,\tilde{x}, y, \tilde{y}, \xi,\eta) = 
  \int\limits_0^1 \nu_{1,h}(r x+(1-r) \tilde{x},\xi)\, 
  \nu_{2,h}(r y+(1-r) \tilde{y},\eta)\ d r,
\end{align*}
we see that it is a multiple symbol in $S^{m_1,m_2}(\R^{2n}\times
\R^n\times \R^{2n}\times \R^n)$ and the term in (\ref{eq: comp
3}) can be written as
\begin{align*}
 h \mint{4} e^{i\Sigma(z,\zeta,t,\tau,\xi)}  
    \nu_h (x+z,\tilde{x}, y+t, \tilde{y}, \zeta,\tau)
    \ d z\ d \zeta\ d t\ d \tau
    \left.\mstrut{0.4cm}\right|_{\tilde{x}=\tilde{y}=y=x},
\end{align*}
 Applying Theorem 2.2.5 in
  \cite{Kumano-go:81} twice (once for the integrations \wrt $z$ and
  $\tau$, a second time for the integrations \wrt $t$ and $\zeta$,
  recalling that $\Sigma(z,\zeta,t,\tau,\xi) = 2
(\sympl{t}{\tau-\xi}{z}{\zeta-\xi})$) we obtain
  that the last term in (\ref{eq: comp 2}) is of the form
  $h \lambda_h^{(1)} (x,\xi)$, where $\lambda_h^{(1)} (x,\xi)$ is in
  $S^{l+1}$ uniformly in $h$. 

Similarly, by Lemma~\ref{lemma: getting h
  out},  we write
  \begin{align*}
  \d_{x_j} p_h = h^{\hf} \tilde{\nu}_{h (j)}, \quad 
  \d_{\xi_j} p_h = h^{\hf} \tilde{\nu}_{h}^{(j)},
\end{align*}
with $\tilde{\nu}_{h (j)}$ in $S^1$ and
$\tilde{\nu}_{h}^{(j)}$ in $S^0$ uniformly in $h$. The same
reasoning as above yields the last term in (\ref{eq: comp 2}) is of the form
  $h^\hf \lambda_h^{(0)} (x,\xi)$, where $\lambda_h^{(0)}  (x,\xi)$ is in
  $S^{l}$ uniformly in $h$.

To now treat the last equality in (\ref{eq: r p}) we use
  Proposition~\ref{prop: Weyl composition} for $k=1$:
  \begin{align}
    \label{eq: composition ordre 2}
    r_h\ &\#^w p_h (x,\xi) = (r_h p_h)(x,\xi)
    +\frac{1}{2i} \{r_h,p_h\}(x,\xi)\\
    &+ \frac{\pi ^{-2n}}{(2i)^2}\int\limits_0^1 (1-r) 
    \!\!\!\sum_{1\leq j,k\leq n} \mint{4} e^{i\Sigma(z,\zeta,t,\tau,\xi)} 
    \left( \mstrut{0.3cm} \right.\!\!
    \d^2_{\xi_j\xi_k} r_h(x+r z,\zeta) 
    \d^2_{x_j x_k} p_h(x+r t,\tau)
    \nonumber \\& \hspace*{1cm}
    -2 \d^2_{x_j\xi_k} r_h(x+r z,\zeta) \d^2_{\xi_j x_k} p_h(x+r
    t,\tau)
    +\d^2_{x_j x_k} r_h(x+r z,\zeta) \d^2_{\xi_j\xi_k} p_h(x+r t,\tau)
    \!\left. \mstrut{0.3cm} \right)
    \ d r \ d z\ d \zeta\ d t\ d \tau.
    \nonumber
  \end{align}
  Here, by Lemma~\ref{lemma: getting h
  out},  we  write
  \begin{align*}
  \d^2_{x_j x_k} p_h = h \nu_{h (j,k)}, \quad 
  \d^2_{\xi_j x_k } p_h = h \nu_{h (k)}^{(j)},\quad 
  \d^2_{\xi_j \xi_k } p_h = h \nu_{h}^{(j,k)},
\end{align*}
  where $\nu_{h (j,k)}$, $\nu_{h (k)}^{(j)}$, and $\nu_{h}^{(j,k)}$\
  are respectively in $S^2$, $S^1$, and $S^0$ uniformly in $h$. We can then
  conclude as above with Theorem 2.2.5 in \cite{Kumano-go:81} and find
  the last term in (\ref{eq: composition ordre 2})
  of the form $h \tilde{\lambda}_h^{(0)}$ with $\tilde{\lambda}^{(0)}_h$
  in $S^{l}$ uniformly in $h$. 
  \hfill \qedsymbol
 
\section{Effect of a change of variables}
\label{sec: appendix change variables}
\subsection{Pseudodifferential calculus results}

We shall be interested in transformation formulae for Weyl symbols
under a change of variables and apply them to the particular symbols
we consider in Section~\ref{sec: manifold}.  We let $X$ and
$\tilde{X}$ be two open subsets of $\R^n$ and let $\kappa: X \to
\tilde{X}$ be a diffeomorphism.  We shall study the effect of the
change of variables $x \mapsto \kappa(x)$ on the symbol $\chi\, \#^w
p_h\, \#^w \chi$ in the Weyl quantization, where $p_h = e^{-h q}$,
with $q$ satisfying the assumptions made in Section~\ref{sec:
  manifold} above and $\chi \in \Cinfc(X)$.

We first consider general amplitudes before specializing to Weyl
symbols.  Let $a(x,y,\xi)$ be the amplitude in $S^m(X\times X
\times\R^n)$ of $A \in \Psi^m(X)$ whose kernel is compactly supported.
In particular, below,  we shall consider $a(x,y,\xi)$ to be of the form
$\chi(x)\, \chi(y)\, \tilde{a}(x,y,\xi)$, with $\tilde{a} \in S^m(X\times X
\times\R^n)$.  With $\zeta \in \Cinfc(\R^n)$ equal to 1 in a \nhd
of $0$ we set
\begin{align*}
  a_0 (x,y,\xi) = \zeta(x-y)\, a(x,y,\xi), \quad \text{and}\quad
  a_\infty(x,y,\xi) = (1-\zeta(x-y))\, a(x,y,\xi).
\end{align*}
If we set $A_\kappa = (\kappa^{-1})^\ast \circ A \circ \kappa^\ast$,
then $A_\kappa \in \Psi^m(\tilde{X})$. In fact, for $\supp(\zeta)$
sufficiently small, $ A_\kappa = A_{0,\kappa} + A_{\infty,\kappa}$,
with $A_{\infty,\kappa} \in \Psi^{-\infty}(\tilde{X})$, and an
amplitude of $A_{0,\kappa}$ is given by \cite{GS:94}
\begin{align}
  \label{eq: change of variable amplitude}
  a_{0,\kappa}(x,y,\xi) = a_0(\kappa^{-1}(x), \kappa^{-1}(y), 
  \transp(\wt{\kappa^{-1}}(x,y))^{-1} \xi) \,
  |\det (\kappa^{-1})'(y)|\, |\det \wt{\kappa^{-1}}(x,y)|^{-1},
\end{align}
where $\wt{\kappa^{-1}}(x,y) = (\wt{\kappa^{-1}_{kl}}(x,y))_{1\leq
  k,l\leq n}$ is defined through
\begin{align*}
  %\label{eq: def kappa tilde}
  \kappa^{-1}_k(x) - \kappa^{-1}_k(y) 
  = \sum_{l=1}^n \wt{\kappa^{-1}_{kl}}(x,y) (x_l - y_l).
\end{align*}
Note that $\wt{\kappa^{-1}}(x,x) = (\kappa^{-1})'(x)$ which implies
that $\wt{\kappa^{-1}}(x,y)$ is indeed invertible in the support of
$a_0$ when $\supp(\zeta)$ is sufficiently small. Note also that
\begin{align}
  \label{eq: diff kappa tilde}
  \d_{x_j} \wt{\kappa^{-1}}(x,y)|_{y=x} = \d_{y_j} \wt{\kappa^{-1}}(x,y)|_{y=x}
  = \hf \d_{x_j}(\kappa^{-1})'(x), \quad j=1,\dots,n.
\end{align}
Note that, for the operator $A_\infty$, we can regularize its kernel
by integration by parts and use the amplitude 
\begin{align}
  \label{eq: regularized a infty}
  a_{\infty}^{(k)} (x,y,\xi) = L^k a_\infty (x,y,\xi), \quad \text{with}\ 
  L = \frac{i}{|x-y|^2} \sum_{i=1}^n (x_i-y_i) \d_{\xi_i}, \quad k\in\N,
\end{align}
in place of $a_\infty(x,y,\xi)$.

By Proposition~\ref{prop: amplitude to Weyl symbol} (with $k=1$), the Weyl symbol of $A_{0,\kappa}$ is given by
\begin{align}
  \label{eq: amplitude to Weyl}
  \al_\kappa(x,\xi) &= \left.e^{\frac{i}{2} (\inp{D_y}{D_\xi}- \inp{D_x}{D_\xi})}
    a_{0,\kappa}(x,y,\xi)\right|_{y=x}
    = \pi^{-n} \iint e^{2i \inp{z}{\zeta-\xi}}  a_{0,\kappa}(x+z,x-z,\zeta)
    \, d z\ d \zeta\\
    &= \underbrace{\phantom{\frac{1}{2}}\!\!\!\!a_{0,\kappa}(x,x,\xi)}_{=:\al_{\kappa,0}(x,\xi)} 
    + \underbrace{\frac{i}{2} \left.\inp{\d_{x} - \d_{y}}{\d_{\xi}}\, 
    a_{0,\kappa}(x,y,\xi)\right|_{y=x}}_{=:\al_{\kappa,1}(x,\xi)} \nonumber \\
  &\quad + \pi^{-n} \left(\frac{i}{2}\right)^2 
  \!\int\limits_0^1\!\! \iint e^{2i \inp{z}{\zeta-\xi}} (1-r) 
  (\inp{\d_{x} - \d_{y}}{\d_{\zeta}}^2 
  a_{0,\kappa})(x+r z,y-r z,\zeta)\, d r \ d z\ d \zeta
  \!\left.\mstrut{0.35cm}\right|_{y=x}.
  \nonumber
\end{align}

 We now specialize to an amplitude $\tilde{a}(x,y,\xi)$
given by the Weyl quantization, i.e.,
\begin{equation*}
  a(x,y,\xi) = \chi(x)\, \chi(y)\, b((x+y)/2,\xi).
\end{equation*}
To simplify some notation we set $L= \kappa^{-1}$.
The symbol $\al_{\kappa,0}(x,\xi)$ is then given by
\begin{align}
  \label{eq: al 0 kappa}
  \al_{\kappa,0}(x,\xi) = \chi(L(x))^2\, b(L(x), 
  \transp \kappa'(L(x)) \xi).
\end{align}
%%%%%%%%%%%%%%%%%%%%%%%%
% lemma                %
%%%%%%%%%%%%%%%%%%%%%%%%
\begin{lemma}
\label{lemma: al 1 kappa}
The symbol $\al_{\kappa,1}(x,\xi)$ is given by
\begin{align*}
  \al_{\kappa,1}(x,\xi) = \frac{i}{2}\, \chi(L(x))^2
  \sum_{k=1}^n f_k(x) \,
  (\d_{\xi_k} b)(L(x),\transp \kappa'(L(x)) \xi), 
\end{align*}
where
\begin{align*}
  f_k(x) =  \sum_{l=1}^n \d_{x_l}( \kappa'_{l k}(L(x))
  =
  \!\!\sum_{1\leq m,l\leq n} (\d^2_{x_k,x_l} \kappa_m)(L(x))\: 
  (\d_{x_m} L_l)(x).
\end{align*}
\end{lemma}
%%%% proof
\begin{proof}
  From the definition of $\al_{\kappa,1}$ in (\ref{eq: amplitude to Weyl}), 
  and (\ref{eq: change of variable amplitude}) we have
  \begin{multline*}
    \al_{\kappa,1}(x,\xi) = \frac{i}{2} \inp{\d_{x} - \d_{y}}{\d_{\xi}}\, 
    \left( \chi(L(x)) \, \chi(L(y))\, b((L(x)+L(y))/2, 
      \transp(\wt{\kappa^{-1}}(x,y))^{-1} \xi) \right. \\
    \times\left.\left.|\det (L)'(y)|\, |\det \wt{\kappa^{-1}}(x,y)|^{-1}
      \right)\right|_{y=x},
  \end{multline*}
  where we have used that $\zeta$ is equal to one in a \nhd of the
  origin. From (\ref{eq: diff kappa tilde}), we see that we need not
  take into account the spatial differentiations acting on the terms
  $\transp(\wt{\kappa^{-1}}(x,y))^{-1}$. Similarly the spatial
  differentiations acting on the cut-off functions $\chi(L(x))$ and
  $\chi(L(y))$ cancel each other, and so do the spatial
  differentiations acting on the first variable of the symbol
  $b$. Note also that the absolute values for the last two terms can
  be removed before differentiation since their product yields $1$ in
  the case $y=x$.  To simplify the notation we set $M =
  \wt{\kappa^{-1}}(x,x)$. We thus obtain
  \begin{align*}
    \al_{\kappa,1}(x,\xi) =  - \frac{i}{2} \sum_{1\leq j,k\leq n} \,
    \chi(L(x))^2\,  (\d_{\xi_k} b)(L(x), \transp \kappa'(L(x)) \xi) 
    \, (M^{-1})_{j k}\,  
    (\d_{x_j} \det (L'(x)))\, (\det M)^{-1}.
  \end{align*}
  From the multi-linearity of the determinant we find that 
  \begin{align*}
    (\d_{x_j} \det (L'(x)))\,  (\det M)^{-1}
    =\sum_{1\leq p,l\leq n} \d_{x_j} L'_{pl}(x)\, \kappa'_{l p}(L(x)),
  \end{align*}
  which yields
   \begin{align*}
     f_k(x) &= - \sum_{1\leq j,p,l\leq n} \kappa'_{j k}(L(x)) \,
     \left(\d_{x_j} L'_{pl}(x)\right)\, \kappa'_{l p}(L(x)) 
     = -\sum_{1\leq j,p,l\leq n}  \kappa'_{l p}(L(x))\,
     \left(\d_{x_l} L'_{p j}(x)\right)\, \kappa'_{j k}(L(x)) \\
     &= \sum_{l=1}^n  \d_{x_l} (\kappa'_{l k}(L(x))),
   \end{align*}
   since $\kappa'(L(x))\, L'(x) = \id_{\tilde{X}}$.
\end{proof}

\subsection{Application to the operator $\bld{\chi \circ p_h^w(x,D_x)\circ \chi}$}
We use the notation introduced above.  In the case $b = e^{-h q} =
p_h$ then $a(x,y,\xi) = \chi(x)\, \chi(y)\, e^{-h q((x+y)/2,\xi)}$ is
an amplitude for the operator $A=\chi \circ p_h^w(x,D_x)\circ \chi$
with Weyl symbol $\al = \chi\, \#^w p_h\, \#^w \chi$. Making use of
the form of the amplitude $a_\infty^{(k)}$ in (\ref{eq: regularized a
  infty}), we see that $A_{\infty,\kappa} = h
\tilde{A}_{\infty,\kappa}$ with $ \tilde{A}_{\infty,\kappa}$ in
$\Psi^0(\tilde{X})$ uniformly in $h$, using Lemma~\ref{lemma: getting
  h out}. 

We now focus on the operators $A_0$ and $A_{0,\kappa}$.  From
(\ref{eq: al 0 kappa}) and Lemma~\ref{lemma: al 1 kappa}, the
expression of the remainder term in (\ref{eq: amplitude to Weyl}) and
using Lemma~\ref{lemma: getting h out} we obtain
\begin{align}
  \label{eq: al kappa}
  \al_\kappa(x,\xi)  = \chi(L(x))^2 p_h(L(x), \transp \kappa'(L(x)) \xi) 
  \left(\mstrut{0.4cm}\right.\!\!
    1 - h \frac{i}{2} \sum_{k=1}^n f_k(x) \,
    (\d_{\xi_k} q)(L(x),\transp \kappa'(L(x)) \xi)  
  \!\!\left.\mstrut{0.4cm}\right) + h \tilde{\al}_\kappa,
\end{align}
with $\tilde{\al}_\kappa$ in $S^0$ uniformly in $h$.  Similarly, if we
denote by $q_\kappa$ the Weyl symbol of $(\kappa^{-1})^\ast \circ
q^w(x,D_x) \circ \kappa^\ast$, we have
\begin{align}
  \label{eq: weyl change of variable}
  q_\kappa(x,\xi)  = q(L(x), \transp \kappa'(L(x)) \xi) 
  + \frac{i}{2} \sum_{k=1}^n f_k(x) \,
    (\d_{\xi_k} q)(L(x),\transp \kappa'(L(x)) \xi)+ \tilde{q}_\kappa,
\end{align}
where $\tilde{q}_\kappa \in S^0$.  We now prove that after the change
of variables $x \mapsto \kappa(x)$, for the operator $ \chi \circ
p_h^w(x,D_x) \circ \chi = \al^w(x,D_x)$, we may use the symbol
$\chi(L(x))\, \#^w e^{-h q_\kappa(x,\xi)} \, \#^w \chi(L(x))$ in place
of $\al_{\kappa}(x,\xi)$, the pullback of $\al$ in the Weyl
quantization, yet remaining within a first-order precision \wrt to the
small parameter $h$.
%%%%%%%%%%%%%%%%%%%%%%%%
% theorem              %
%%%%%%%%%%%%%%%%%%%%%%%%
\begin{lemma}
  \label{lemma: change of variable chi ph chi}
  We set $\hat{p}_h(x,\xi) = e^{-h q_\kappa(x,\xi)}$. We have 
  \begin{align*}
    \left( ((\kappa^{-1})^\ast \chi)\, \#^w  \hat{p}_h \, 
    \#^w ((\kappa^{-1})^\ast \chi)\right)(x,\xi)
    - \al_{\kappa}(x,\xi) = h \lambda_h (x,\xi),
  \end{align*}
  where $\lambda_h$ is in $S^0$ uniformly in $h$.
\end{lemma}
%%%% proof
\begin{proof}
  We set $\nu(x,\xi) = \frac{i}{2} \sum_{k=1}^n f_k(x) (\d_{\xi_k} q)
  (L(x),\transp \kappa'(L(x)) \xi) $.   Making use of (\ref{eq: weyl
  change of variable}), we write
\begin{align*}
  \hat{p}_h(x,\xi) &= p_h (L(x),\transp\kappa'(L(x))\xi) 
  e^{- h \nu(x,\xi)}
  e^{-h \tilde{q}_\kappa(x,\xi)}\\
  & =  p_h (L(x),\transp\kappa'(L(x))\xi) 
  \left(\mstrut{0.4cm}\right. \!\!1 - h \nu(x,\xi) + (h\nu(x,\xi))^2  
  \int\limits_0^1 e^{-r h \nu(x,\xi)} (1-r )\: d r\!\!
  \left.\mstrut{0.4cm} \right) (1 + h\mu_1(x,\xi)),
\end{align*}
by two Taylor formulae, where $\mu_1$ is in $S^0$ uniformly in
$h$. From Lemmata~\ref{lemmaM: property q} and \ref{lemma: getting h
  out} we obtain that
\begin{align*}
   p_h (L(x),\transp\kappa'(L(x))\xi)   (h\nu(x,\xi))^2  
  \int\limits_0^1 e^{-r h \nu(x,\xi)} (1-r )\: d r = h \mu_2(x,\xi),
\end{align*}
with $\mu_2$ in $S^0$ uniformly in $h$.
From (\ref{eq: al kappa}) we hence obtain
\begin{align*}
  \al_\kappa(x,\xi) - \chi(L(x))^2 \hat{p}_h (x,\xi) = h \mu_3, 
\end{align*}
with $\mu_3$ in $S^0$ uniformly in $h$.  We conclude the proof with
the following lemma since $\hat{p}_h$ and $p_h$ are of the same nature.
\end{proof}
%%%%%%%%%%%%%%%%%%%%%%%%
% lemma                %
%%%%%%%%%%%%%%%%%%%%%%%%
\begin{lemma}
  \label{lemma: symbol conjugation chi}
  Let $\phi \in \Cinfc(X)$. We then have
  \begin{align*}
    \phi\, \#^w p_h\, \#^w \phi - \phi^2 p_h = h \lambda_h,
  \end{align*}
  where $\lambda_h$ is in $S^0$ uniformly in $h$.
\end{lemma}
%%%% proof
\begin{proof}
  Since $\phi(x) \phi(y) p_h((x+y)/2,\xi)$ is an amplitude for the operator
  with Weyl symbol $\phi\, \#^w p_h\, \#^w \phi$, by (\ref{eq:
    amplitude to Weyl}) we obtain
  \begin{align*}
    (\phi\, \#^w p_h\, \#^w \phi) (x,\xi) &= 
    \pi^{-n} \iint e^{2i \inp{z}{\xi-\zeta}} \phi(x-z)\,  \phi(x+z)\,
    p_h(x,\zeta)\, d z \ d \zeta\\
    &=\phi^2(x) p_h(x,\xi)
    -\frac{1}{4} \pi^{-n}
    \!\!\!\!\sum_{1\leq j,k\leq n} \int\limits_0^1\!\! \iint 
    (1-r)  e^{2i \inp{z}{\xi-\zeta}}  
    \left(-\phi(x+r z)\, \d^2_{x_j,x_k}\phi(x-r z)\right. \\
    &\hspace*{0.5cm}-\left. 2 \d_{x_j} \phi(x+r z)\, \d_{x_k}\phi(x-r z)\, 
    + \d^2_{x_j,x_k} \phi(x+r z)\, \phi(x-r z)\right)\, 
    \d^2_{\xi_j,\xi_k} p_h(x,\zeta)\,
    \ d r\ d z \ d \zeta.
  \end{align*}
  We then conclude as in the proof of
  Proposition~\ref{prop:composition r ph} in Appendix~\ref{appendix:
  composition} by using Lemma~\ref{lemma: getting h out} and Theorem
  2.2.5 in \cite{Kumano-go:81}.
\end{proof}

% acknowledgement
\paragraph{\normalsize \bfseries Acknowledgements:} 
We wish to thank L.~Robbiano for discussions on pullbacks of Weyl
symbols.  This work was initiated while the second author was visiting
the Institute of Mathematics at the University of Tsukuba.  Part of
this work was done when the second author was on a research leave at
Universit\'e Pierre et Marie Curie, Laboratoire Jacques-Louis Lions,
CNRS UMR 7598, Paris, France. The second author was partially
supported by l'Agence Nationale de la Recherche under grant
ANR-07-JCJC-0139-01.

% references
\bibliographystyle{amsalpha}
\bibliography{references}

\end{document}